\documentclass[letterpaper,reqno,11pt]{amsart}
\usepackage[margin=1.2in]{geometry}

\usepackage{amssymb, enumerate}
\usepackage[all]{xy}
\usepackage{hyperref}
\hypersetup{colorlinks=true}

\theoremstyle{plain}
\newtheorem{thm}{Theorem}[section] 
\newtheorem{cor}[thm]{Corollary}
\newtheorem{prop}[thm]{Proposition}
\newtheorem{conj}[thm]{Conjecture}
\newtheorem{lem}[thm]{Lemma}
\newtheorem*{mainthm1}{Theorem~A}
\newtheorem*{mainthm2}{Theorem~B}

\theoremstyle{definition} 
\newtheorem{defn}[thm]{Definition}
\newtheorem{eg}[thm]{Example} 

\theoremstyle{remark}
\newtheorem{rem}[thm]{Remark}
\newtheorem{ques}[thm]{Question}

\def\ge{\geqslant}
\def\le{\leqslant}
\def\phi{\varphi}
\def\epsilon{\varepsilon}

\def\to{\longrightarrow}
\def\mapsto{\longmapsto}
\def\into{\lhook\joinrel\longrightarrow}

\def\mod{\operatorname{\,mod}}

\newcommand{\sO}{\mathcal{O}}

\newcommand{\F}{\mathbb{F}}
\newcommand{\N}{\mathbb{N}}
\newcommand{\Q}{\mathbb{Q}} 
\newcommand{\R}{\mathbb{R}} 
\newcommand{\Z}{\mathbb{Z}}

\newcommand{\ba}{\mathfrak{a}}
\newcommand{\m}{\mathfrak{m}}
\newcommand{\n}{\mathfrak{n}}
\newcommand{\p}{\mathfrak{p}}
\newcommand{\q}{\mathfrak{q}}

\def\Hom{\operatorname{Hom}}
\def\Spec{\operatorname{Spec}}
\def\Proj{\operatorname{Proj}}
\def\Supp{\operatorname{Supp}}

\title{A Gorenstein criterion for strongly $F$-regular and\\ log terminal singularities}

\author{Anurag K. Singh}
\address{Department of Mathematics, University of Utah, 155 South 1400 East, Salt Lake City, UT~84112, USA}
\email{singh@math.utah.edu}

\author{Shunsuke Takagi}
\address{Graduate School of Mathematical Sciences, University of Tokyo, 3-8-1 Komaba, Meguro-ku, Tokyo 153-8914, Japan}
\email{stakagi@ms.u-tokyo.ac.jp}

\author{Matteo Varbaro}
\address{Dipartimento di Matematica, Universit\`a di Genova, Via Dodecaneso 35, I-16146 Genova, Italy}
\email{varbaro@dima.unige.it}

\thanks{A.K.S.~was partially supported by NSF grants DMS~1162585 and DMS~1500613, S.T.~by JSPS KAKENHI 26400039, and M.V.~by Geometria delle Variet\`a Algebriche PRIN 2010S47ARA 003. The third author worked on part of this project while he was visiting the Max Planck Institute for Mathematics of Bonn.\newline \indent The authors are grateful to Alessandro De Stefani, Yoshinori Gongyo, Hirotaka Onuki, Kei-ichi Watanabe, Ken-ichi Yoshida, and the referees for helpful comments.}

\dedicatory{Dedicated to Professor~Craig~Huneke on the occasion of his sixty-fifth birthday.}

\begin{document}

\begin{abstract}
A conjecture of Hirose, Watanabe, and Yoshida offers a characterization of when a standard graded strongly $F$-regular ring is Gorenstein, in terms of an $F$-pure threshold. We prove this conjecture under the additional hypothesis that the anti-canonical cover of the ring is Noetherian. Moreover, under this hypothesis on the anti-canonical cover, we give a similar criterion for when a normal $F$-pure (resp.~log canonical) singularity is quasi-Gorenstein, in terms of an $F$-pure (resp.~log canonical) threshold.
\end{abstract}

\maketitle
\markboth{A.~K.~SINGH, S.~TAKAGI and M.~VARBARO}{A GORENSTEIN CRITERION FOR STRONGLY $F$-REGULAR RINGS}

\section{Introduction}

Let $R$ be an $F$-pure domain of positive characteristic, and $\ba$ a nonzero proper ideal. The $F$-pure threshold $\mathrm{fpt}(\ba)$ was defined by Watanabe and the second author of this paper \cite{TW}; it may be viewed as a positive characteristic analogue of the log canonical threshold, and is an important measure of the singularities of the pair $(\Spec R, V(\ba))$. 
For example, a local ring~$(R,\m)$ is regular if and only if $\mathrm{fpt}(\m)>\dim R-1$.

We say that $R$ is strongly $F$-regular if $\mathrm{fpt}(\ba)>0$ for each nonzero proper ideal~$\ba$ of $R$. 
It is well-known that each strongly $F$-regular ring is Cohen-Macaulay and normal; it is then natural to ask:
when is a strongly $F$-regular ring Gorenstein? 
Toward answering this in the graded context, Hirose, Watanabe, and Yoshida proposed the following:

\begin{conj}{\cite[Conjecture~1.1~(2)]{HWY}}
\label{conj:HWY}
Let $R$ be a standard graded strongly $F$-regular ring, with $R_0$ an $F$-finite field of characteristic $p>0$. Let $\m$ be the unique homogeneous maximal ideal of $R$. 
Then~$\mathrm{fpt}(\m)=-a(R)$ if and only if $R$ is Gorenstein. 
\end{conj}

We prove that the conjecture holds for many classes of (not necessarily strongly $F$-regular) $F$-pure normal standard graded rings: 

\begin{mainthm1}[Corollaries~\ref{cor:conclusion I},~\ref{cor:final}]
Let $R$ be an $F$-pure, normal, standard graded ring, with~$R_0$ an $F$-finite field of characteristic $p>0$. Let $\m$ denote the homogeneous maximal ideal of $R$. 
Set $X=\Spec R$, and suppose that the anti-canonical cover $\bigoplus_{n \ge 0} \sO_X(-nK_X)$ of~$X$ is a Noetherian ring.
Then $\mathrm{fpt}(\m)=-a(R)$ if and only if $R$ is quasi-Gorenstein. 
\end{mainthm1}

Note that a ring is Gorenstein if and only if it is Cohen-Macaulay and quasi-Gorenstein. Under the hypotheses of Theorem~A, the anti-canonical cover $\bigoplus_{n \ge 0} \sO_X(-nK_X)$ is known to be Noetherian in each of the following cases:
\begin{enumerate}[\quad\rm(1)]
\item $R$ is $\Q$-Gorenstein, 
\item $R$ is a semigroup ring,
\item $R$ is a determinantal ring, 
\item $R$ is a strongly $F$-regular ring of dimension at most three,
\item $R$ is a four-dimensional strongly $F$-regular ring, of characteristic $p>5$. 
\end{enumerate}

We give two proofs of Theorem~A: the first has the advantage that it can be adapted to obtain results in the local setting, e.g., Theorem~\ref{thm:anti-canonical}. The second proof, while limited to the graded context, provides a technique for computing the numerical invariants at hand; see Proposition~\ref{prop:determinantal} for the case of determinantal rings.
We describe the two techniques, after recalling some definitions: 

Recall that a ring $R$ of prime characteristic $p$ is called $F$-finite if the Frobenius map $F\colon R \to F_*R$ is a finite map. 
Let $R$ be a local or standard graded $F$-finite domain of characteristic $p>0$, and suppose that $R$ is $F$-pure, i.e., the $e$-th iterated Frobenius map $F^e \colon R \to F^e_*R$ with $x\mapsto F^e_*x^{p^e}$ splits as an $R$-linear map for each $e\ge 1$. 
Given a nonzero ideal $\ba \subsetneq R$, and an integer $e\ge 1$, set $\nu_e(\ba)$ to be the largest integer~$r \ge 0$ such that there exists a nonzero element $c$ in $\ba^r$ for which the composite map 
\[
R \xrightarrow{F^e} F^e_*R \xrightarrow{\times F^e_*c} F^e_*R,\quad\text{ where }\quad x \mapsto F^e_*x^{p^e} \mapsto F^e_*(cx^{p^e}),
\]
splits as an $R$-module homomorphism. Then, $\mathrm{fpt}(\ba)$ is defined to be $\lim_{e\to\infty}\nu_e(\ba)/p^e$.

The first proof of Theorem~A uses an invariant $c(\ba)$ that was originally introduced in \cite{ST}: 
Given a nonzero ideal $\ba \subsetneq R$, this invariant is defined in terms of the Grothendieck trace of the iterated Frobenius map $\mathrm{Tr}^e \colon F^e_*\omega_R \to \omega_R$, where~$\omega_R$ is the canonical module of $R$.
For an $F$-pure normal graded ring $(R, \m)$, one has
\[
\mathrm{fpt}(\m) \le c(\m) \le -a(R),
\]
with equality holding when $R$ is a quasi-Gorenstein standard graded ring; see Propositions~\ref{prop:a_i} and~\ref{prop:fit=-a}. 
Thus, it suffices to show that if $\mathrm{fpt}(\m)=c(\m)$, then $R$ is quasi-Gorenstein. 
Generalizing the argument of \cite[Theorem~2.7]{TW}, we are indeed able to prove this when the anti-canonical cover of $R$ is Noetherian. 
We also use the invariant $c(\m)$ in answering another question of Hirose, Watanabe, and Yoshida, \cite[Question~6.7]{HWY}; see Corollary~\ref{cor:toric charp}.

Our second proof uses the so-called Fedder-type criterion: Writing the standard graded ring $R$ as~$S/I$, for $S$ a polynomial ring and $I$ a homogeneous ideal, we characterize~$\nu_e(\m)$ in terms of the ideal $I^{[p^e]}:_SI$, and use this to show that $-\nu_e(\m)$ equals the degree of a minimal generator of the $(1-p^e)$-th symbolic power of $\omega_R$, see Theorem~\ref{theorem:omega}. 
Using this, we give explicit computations of $\mathrm{fpt}(\m)$ in many situations, e.g., for determinant rings and for $\Q$-Gorenstein rings, see Propositions~\ref{prop:determinantal} and~\ref{prop:fptQGor}. We also prove that if $(R,\m)$ is a $\Q$-Gorenstein normal domain, with index coprime to $p$, then the pair $(R,\m^{\mathrm{fpt}(\m)})$ is sharply $F$-pure; see Proposition~\ref{prop:critical value}.

Thus far we have discussed singularities in positive characteristic; we also prove an analogous result in characteristic zero. 
de~Fernex-Hacon \cite{dFH} extended the definition of log terminal and log canonical singularities to the non-$\Q$-Gorenstein setting, which can be regarded as the characteristic zero counterparts of strongly $F$-regular and $F$-pure rings. 
Using their definition, we formulate a characteristic zero analogue of Theorem~A as follows:

\begin{mainthm2}[Corollary~\ref{cor:char 0 main thm}] 
Let $R$ be a standard graded normal ring, with $R_0$ an algebraically closed field of characteristic zero. Set $\m$ to be the homogeneous maximal ideal of $R$. 
Assume that $X:=\Spec R$ has log canonical singularities in the sense of de~Fernex-Hacon;~set 
\[
\mathrm{lct}(\m)=\sup\{t \ge 0 \mid (X, \m^t) \textup{ is log canonical in the sense of de~Fernex-Hacon}\}.
\]
\begin{enumerate}[\quad\rm(1)]
\item Then $\mathrm{lct}(\m) \le -a(R)$.
\item Suppose, in addition, that the anti-canonical cover $\bigoplus_{n \ge 0}\sO_X(-nK_X)$ is Noetherian. 
Then $\mathrm{lct}(\m)=-a(R)$ if and only if $R$ is quasi-Gorenstein. 
\end{enumerate}
\end{mainthm2}

We remark that in the situation of Theorem~B, the anti-canonical cover is Noetherian whenever $X$ has log terminal singularities in the sense of de~Fernex-Hacon, or if $R$ is $\Q$-Gorenstein. Thus, Theorem~B gives an affirmative answer to a conjecture of De~Stefani-N\'u\~nez-Betancourt \cite[Conjecture~6.9]{dSNB}. 

In order to prove Theorem~B, we introduce a new invariant $\mathrm{d}(\ba)$ for an ideal $\ba$ of a normal variety $X$ with Du Bois singularities, in terms of a variant of multiplier modules, see Definition~\ref{def:db threshold}.
We are then able to employ the same strategy as in the first proof of Theorem~A, using $\mathrm{d}(\ba)$ in place of $c(\ba)$. 

\medskip

Throughout this paper, all rings are assumed to be Noetherian (except possibly for anti-canonical covers), commutative, with unity. By a \textit{standard graded ring}, we mean an $\N$-graded ring $R=\bigoplus_{n \ge 0}R_n$, with $R_0$ a field, such that $R$ is generated as an $R_0$-algebra by finitely many elements of $R_1$.

\section{Preliminaries on $F$-singularities}

In this section, we briefly review the theory of $F$-singularities. 
In order to state the definitions, we first introduce the following notation: 

Let $R$ be a ring of prime characteristic $p>0$. 
We denote by $R^{\circ}$ the set of elements of $R$ that are not in any minimal prime ideal.
Given an $R$-module $M$ and $e \in \N$, the $R$-module $F^e_*M$ is defined by the following two conditions: (i)~$F^e_*M=M$ as an abelian group, and (ii)~the $R$-module structure of $F^e_*M$ is given by $r \cdot x:=r^{p^e}x$ for $r \in R$ and $x \in F^e_*M$. 
We write elements of $F^e_*M$ in the form $F^e_*x$ with $x \in M$. 
The $e$-th iterated \textit{Frobenius map} is the $R$-linear map $F^e \colon R \to F^e_*R$ sending $x$ to $F^e_*x^{p^e}$. 
We say that $R$ is \textit{$F$-finite} if the Frobenius map is finite, that is, $F^1_*R$ is a finitely generated $R$-module. 
When $(R,\m)$ is local, the $e$-th iterated Frobenius map $F^e \colon R \to F^e_*R$ induces a map $F^e_{H^i_{\m}(R)} \colon H^i_{\m}(R) \to H^i_{\m}(R)$ for each~$i$. 

We recall the definition of classical $F$-singularities:

\begin{defn}
\label{def:classical F-sing}
Let $R$ be an $F$-finite reduced ring of prime characteristic $p>0$. 
\begin{enumerate}[\quad\rm(1)]
\item We say that $R$ is \textit{$F$-pure} if the Frobenius map $R \to F_*R$ splits as an $R$-linear map. 

\item We say that $R$ is \textit{strongly $F$-regular} if for every $c \in R^{\circ}$, there exists a power $q=p^e$ of $p$ such that the $R$-linear map $R \to F^e_*R$ sending $1$ to $F^e_*c$ splits. 

\item When $(R, \m)$ is local, we say that $R$ is \textit{$F$-injective} if $F_{H^i_{\m}(R)} \colon H^i_{\m}(R) \to H^i_{\m}(R)$ is injective for each $i$. In general, we say that $R$ is $F$-injective if the localization $R_{\m}$ is $F$-injective for each maximal ideal $\m$ of $R$.

\item When $(R, \m)$ is local, we say that $R$ is \textit{$F$-rational} if $R$ is Cohen-Macaulay and if for every $c \in R^{\circ}$, there exists $e \in \N$ such that $cF^e_{H^d_{\m}(R)} \colon H^d_{\m}(R) \to H^d_{\m}(R)$ sending $z$ to $cF^e_{H^d_{\m}(R)}(z)$ is injective. 
In general, we say that $R$ is $F$-rational if the localization $R_{\m}$ is $F$-rational for each maximal ideal $\m$ of $R$.
\end{enumerate}
\end{defn}

Next we generalize these to the pair setting, see \cite{HW, Sch, TW}:

\begin{defn}
Let $\ba$ be an ideal of an $F$-finite reduced ring $R$ of prime characteristic $p$ such that $\ba \cap R^{\circ} \ne \emptyset$. 

\begin{enumerate}[\quad\rm(1)]
\item Suppose that $R$ is local. 
For a real number $t \ge 0$, the pair $(R, \ba^t)$ is \textit{sharply $F$-pure} if there exist $q=p^e$ and $c \in \ba^{\lceil t(q-1) \rceil}$ such that the $R$-linear map $R \to F^e_*R$ sending $1$ to $F^e_*c$ splits. 
The pair $(R, \ba^t)$ is \textit{weakly $F$-pure} if there exist infinitely many $e \in \N$ and associated elements $c_e \in \ba^{\lfloor t(p^e-1) \rfloor}$ such that each $R$-linear map $R \to F^e_*R$ sending $1$ to $F^e_* c_e$ splits. 

When $R$ is not local, $(R, \ba^t)$ is said to be sharply $F$-pure (resp. weakly $F$-pure) if the localization $(R_{\m}, \ba_\m^t)$ at $\m$ is sharply $F$-pure (resp. weakly $F$-pure) for every maximal ideal $\m$ of $R$. 

\item Suppose that $R$ is $F$-pure. Then the \textit{$F$-pure threshold} $\mathrm{fpt}(\ba)$ of $\ba$ is defined as
\[
\mathrm{fpt}(\ba)=\sup\{t \in \R_{\ge 0} \mid (R, \ba^t) \textup{ is weakly $F$-pure}\}.
\]

\item Suppose that $R$ is a normal local domain and $\Delta$ is an effective $\Q$-divisor on $X:=\Spec R$. 
For a real number $t \ge 0$, the pair $((R, \Delta); \ba^t)$ is \textit{sharply $F$-pure} if there exists $q=p^e$ and $c \in \ba^{\lceil t(q-1) \rceil}$ such that the $R$-linear map
\[
R \to F^e_*\sO_X(\lceil (q-1) \Delta \rceil)\quad\text{ with }1 \mapsto F^e_*c
\]
splits. 
The pair $((R, \Delta);\ba^t)$ is \textit{weakly $F$-pure} if there exist infinitely many $e \in \N$ and associated elements $c_e \in \ba^{\lfloor t(p^e-1) \rfloor}$ such that each $R$-linear map 
\[
R \to F^e_*\sO_X(\lfloor (p^e-1) \Delta \rfloor)\quad\text{ with } 1 \mapsto F^e_* c_e
\]
splits. 
If, in addition, $\ba=\sO_X$, then we simply say that $(X,\Delta)$ is \textit{$F$-pure}. 

If $(R, \Delta)$ is $F$-pure, then the \textit{$F$-pure threshold} $\mathrm{fpt}(\Delta; \ba)$ of $\ba$ with respect to the pair $(R, \Delta)$ is defined by 
\[
\mathrm{fpt}(\Delta; \ba)=\sup\{t \in \R_{\ge 0} \mid ((R, \Delta); \ba^t) \textup{ is weakly $F$-pure}\}.
\]
\end{enumerate}
\end{defn}

\begin{rem}
\begin{enumerate}[\rm(1)]
\item Sharp $F$-purity implies weak $F$-purity. 
When $\ba=R$, the sharp $F$-purity and the weak $F$-purity of $(R, \ba^t)$ are equivalent to the $F$-purity of $R$. 
Suppose that $R$ is $F$-pure. 
It is easy to check that if $(R, \ba^t)$ 
is weakly $F$-pure with $t>0$, then $(R, \ba^{t-\epsilon})$ 
is sharply $F$-pure for every $t \ge \epsilon>0$ (cf.~\cite[Lemma~5.2]{Sch}). Thus, 
\[
\mathrm{fpt}(\ba)=\sup\{t \in \R_{\ge 0} \mid (R, \ba^t) \textup{ is sharply $F$-pure}\}.
\]
\item 
Our definition of the $F$-purity of $(R, \Delta)$ coincides with the one in \cite[Definition~2.1]{HW}. 
\end{enumerate}
\end{rem}

The following is a standard application of Matlis duality, which we will use in Section~\ref{sec:char p I}. 

\begin{lem}[{cf.~\cite[Proposition~2.4]{HW}}]
\label{lem:sharp F-purity}
Let $(R, \m)$ be a $d$-dimensional $F$-finite normal local ring of characteristic $p>0$, $\Delta$ be an effective $\Q$-divisor on $X=\Spec R$ and $\ba$ be a nonzero ideal of $R$. 
For any real number $t \ge 0$, the pair $((R, \Delta); \ba^t)$ is weakly $F$-pure if and only if there exist infinitely many $e \in \N$ and associated elements $c_e \in \ba^{\lfloor t(p^e-1) \rfloor}$ such that 
\[
c_eF^e_{X, \Delta} \colon H^d_{\m}(\omega_X) \xrightarrow{F^e_{X, \Delta}} H^d_{\m}(\sO_X(\lfloor p^eK_X+(p^e-1)\Delta \rfloor)) \xrightarrow{\times c_e} H^d_{\m}(\sO_X(\lfloor p^eK_X+(p^e-1)\Delta \rfloor))
\] 
is injective, where $F^e_{X, \Delta}$ is the map induced by the $R$-linear map $R \to F^e_*\sO_X(\lfloor (p^e-1) \Delta \rfloor)$ sending $1$ to $F^e_*1$. 
\end{lem}

The following is a reformulation of the so-called ``Fedder-type criterion," that we will use in Section~\ref{sec:char p II}.

\begin{prop}
\label{prop:Fedder criteria}
Let $S=k[x_1, \dots, x_n]$ be a polynomial ring over an $F$-finite field $k$ of characteristic $p>0$ and $I$ be a homogeneous ideal of $S$. 
Suppose that $R:=S/I$ is $F$-pure. 
Given an $e \in \N$ and a homogeneous ideal $\ba \subset S$ containing $I$ such that $\ba R \cap R^{\circ} \ne \emptyset$, we define the integer $\nu_e(\ba)$ by 
\[
\nu_e(\ba):=\max \{r \ge 0 \mid \ba^r(I^{[p^e]}:I) \not\subset (x_1^{p^e}, \dots, x_n^{p^e})\}.
\]
\begin{enumerate}[\quad\rm(1)]
\item For a real number $t\ge 0$, the pair $(R, (\ba R)^t)$ is sharply (resp. weakly) $F$-pure if and only if $\nu_e(\ba)\ge \lceil (p^e-1)t \rceil$ for some $e$ (resp. $\nu_e(\ba)\ge \lfloor (p^e-1)t \rfloor$ for infinitely many $e$).

\item $\mathrm{fpt}(\ba R)=\lim_{e \to \infty}\nu_e(\ba)/p^e$. 
\end{enumerate}
\end{prop}

\begin{proof}
It follows from \cite[Lemma~3.9]{Ta} (where the criterion for $F$-purity is stated in the local setting, but the same argument works in the graded setting). 
\end{proof}

In order to generalize the definition of $F$-rational and $F$-injective rings to the pair setting, we use the notion of $\ba^t$-tight closure and $\ba^t$-sharp Frobenius closure. 

\begin{defn}
Let $\ba$ be an ideal of a reduced ring $R$ of prime characteristic $p>0$ such that $\ba \cap R^{\circ} \ne \emptyset$, and $t\ge 0$ be a real number. 
\begin{enumerate}[\quad\rm(1)]
\item (\cite[Definition~6.1]{HY})
For an ideal $I \subseteq R$, the $\ba^t$-\textit{tight closure} $I^{*\ba^t}$ of $I$ is defined to be the ideal of $R$ consisting of all elements $x \in R$ for which there exists $c \in R^{\circ}$ such that $c \ba^{\lceil t(q-1) \rceil}x^q \subseteq I^{[q]}$ for all large $q=p^e$. 

\item (\cite[Definition~3.10]{Sch}) 
For an ideal $I \subset R$, the $\ba^t$-\textit{sharp Frobenius closure} $I^{F \sharp \ba^t}$ of $I$ is defined to be the ideal of $R$ consisting of all elements $x \in R$ such that $\ba^{\lceil t(q-1) \rceil}x^q \subseteq I^{[q]}$ for all large $q=p^e$. 

\item Suppose that $(R, \m)$ is local. 
The $\ba^t$-\textit{sharp Frobenius closure} $0^{F \sharp \ba^t}_{H^i_{\m}(R)}$ of the zero submodule in $H^i_{\m}(R)$ is defined to be the submodule of $H^i_{\m}(R)$ consisting of all elements $z \in H^i_{\m}(R)$ such that $\ba^{\lceil t(q-1) \rceil}F^e_{H^i_{\m}(R)}(z)=0$ in $H^i_{\m}(R)$ for all large $q=p^e$. 
\end{enumerate}
\end{defn}

The following technical remark is useful for the study of the invariant $c(\ba)$, which will be introduced in Section~\ref{sec:char p I}. 
\begin{rem}
\label{remark:hara-takagi}
Let $(R, \m)$ be an $F$-finite reduced local ring of characteristic $p>0$. 
Let $\mathrm{Tr}^e \colon F^e_*\omega_R \to \omega_R$ be the $e$-th iteration of the \textit{trace map} on $R$, that is, the $\omega_R$-dual of the $e$-th iterated Frobenius map $F^e \colon R \to F^e_*R$. 
It then follows from an argument similar to the proof of \cite[Lemma~2.1]{HT} that $0^{F \sharp \ba^t}_{H^d_{\m}(R)}=0$ if and only if 
\[
\sum_{e \ge e_0}\mathrm{Tr}^e(F^e_*(\ba^{\lceil t(p^e-1) \rceil}\omega_R))=\omega_R
\]
for every integer $e_0 \ge 0$. 
\end{rem}

\begin{defn}
Let $R$ be an $F$-finite Cohen-Macaulay reduced ring of prime characteristic $p>0$, $\ba$ be an ideal of $R$ such that $\ba \cap R^{\circ} \ne \emptyset$, and $t \ge 0$ be a real number. 
\begin{enumerate}[\quad\rm(1)]
\item (\cite[Definition~6.1]{ST}) When $R$ is local, $(R, \ba^t)$ is said to be \textit{$F$-rational} if $J^{*\ba^t}=J$ for every ideal $J$ generated by a full system of parameters for $R$. 

\item When $R$ is local, $(R, \ba^t)$ is said to be \textit{sharply $F$-injective} if $J^{F\sharp \ba^t}=J$ for every ideal $J$ generated by a full system of parameters for $R$. 
\end{enumerate}

When $R$ is not local, the pair $(R, \ba^t)$ is said to be $F$-rational (resp. sharply $F$-injective) if the localization $(R_{\m}, \ba_{\m}^t)$ at $\m$ is $F$-rational (resp. sharply $F$-injective) for every maximal ideal $\m$ of $R$. 
When $\ba=R$, this definition coincides with the one in Definition~\ref{def:classical F-sing}. 
\end{defn}

We review basic properties of sharply $F$-injective pairs and $F$-rational pairs. 
\begin{lem}
\label{lemma:basic pair}
Let $R$ be an $F$-finite reduced ring of prime characteristic $p>0$, $\ba$ be an ideal of $R$ such that $\ba \cap R^{\circ} \ne \emptyset$, and $t \ge 0$ be a real number. Set $d:=\dim R$.
\begin{enumerate}[\ \rm(1)]
\item Suppose that $(R, \m)$ is Cohen-Macaulay. Then the following are equivalent:
\begin{enumerate}[\ \rm(a)]
\item $(R, \ba^t)$ is sharply $F$-injective (resp. $F$-rational).

\item $J^{F\sharp \ba^t}=J$ (resp. $J^{*\ba^t}=J$) for an ideal $J$ generated by a full system of parameters.

\item $0^{F \sharp \ba^t}_{H^d_{\m}(R)}=0$ (resp. $0^{*\ba^t}_{H^d_{\m}(R)}=0$). 
\end{enumerate}

\item Suppose that $R$ is $F$-rational. 
\begin{enumerate}[\ \rm(a)]
\item There exists a rational number $t_0>0$ such that $(R, \ba^{t_0})$ is sharply $F$-injective. 

\item If $(R, \ba^t)$ is sharply $F$-injective with $t>0$, then $(R, \ba^{t-\epsilon})$ is $F$-rational for every $t \ge \epsilon>0$. 
\end{enumerate}

\item Suppose that $(R, \m)$ is local. If $(R, \ba^t)$ is sharply $F$-pure, then $0^{F \sharp \ba^t}_{H^i_{\m}(R)}=0$ for every~$i$. When $R$ is quasi-Gorenstein, $(R, \ba^t)$ is sharply $F$-pure if and only if $0^{F \sharp \ba^t}_{H^d_{\m}(R)}=0$. 

\item Let $(R, \m) \into (S, \n)$ be a flat local homomorphism of $F$-finite reduced local rings of characteristic $p>0$. 
Suppose that $S/\m S$ is a field which is a separable algebraic extension of $R/\m$. 
Then $0^{F \sharp \ba^t}_{H^d_{\m}(R)}=0$ if and only if $0^{F \sharp (\ba S)^t}_{H^d_{\n}(S)}=0$. 
\end{enumerate}
\end{lem}

\begin{proof}
We may assume throughout that $(R, \m)$ is local. Let $J$ be an ideal generated by a full system of parameters for $R$. 

(1) The $F$-rational case follows from \cite[Lemma~6.3]{ST} and the sharp $F$-injective case follows from an analogous argument. 

(2) First we will show (a). 
Fix a nonzero element $f \in \ba$. 
Since $R$ is $F$-rational, there exists $e_0 \in \N$ such that $fF^{e_0}_{H^d_{\m}(R)} \colon H^d_{\m}(R) \to H^d_{\m}(R)$ is injective. 
Then for each $n \in \N$, the map 
\[
f^{1+p^{e_0}+\cdots+p^{(n-1)e_0}}F^{ne_0}_{H^d_{\m}(R)} \colon H^d_{\m}(R) \to H^d_{\m}(R)
\] 
is also injective. 
Set $t_0=1/(p^{e_0}-1)$ and let $z \in 0^{F \sharp \ba^{t_0}}_{H^d_{\m}(R)}$. 
Since 
\[
f^{1+p^{e_0}+\cdots+p^{(n-1)e_0}}F^{ne_0}_{H^d_{\m}(R)}(z) \in \ba^{t_0(p^{ne_0}-1)}F^{ne_0}_{H^d_{\m}(R)}(z)=0
\] 
for sufficiently large $n$, one has $z=0$ by the injectivity of $f^{1+p^{e_0}+\cdots+p^{(n-1)e_0}}F^{ne_0}_{H^d_{\m}(R)}$. 
It follows that $0^{F \sharp \ba^{t_0}}_{H^d_{\m}(R)}=0$. 

Next we will show (b). 
Let $x \in J^{* \ba^{t-\epsilon}}$. 
Since $1$ is a parameter $\ba^{t-\epsilon}$-test element by \cite[Lemma~6.8]{ST} (see \cite[Definition~6.6]{ST} for the definition of parameter $\ba^{t-\epsilon}$-test elements), 
$\ba^{\lceil t(q-1) \rceil}x^q \subseteq \ba^{\lceil (t-\epsilon)q \rceil}x^q \subseteq J^{[q]}$ for all sufficiently large $q=p^e$. 
Then the sharp $F$-injectivity of $(R, \ba^t)$ implies that $x \in J$, that is, $J^{* \ba^{t-\epsilon}}=J$. 

(3) Let $z \in 0^{F \sharp \ba^t}_{H^i_{\m}(R)}$. 
Since $(R, \ba^t)$ is sharply $F$-pure, there exist a sufficiently large $q=p^e$ and $c \in \ba^{\lceil t(p^e-1) \rceil}$ such that the $R$-linear map $R \to F^e_*R$ sending $1$ to $F^e_*c$ splits, and in particular, 
$cF^e_{H^i_{\m}(R)} \colon H^i_{\m}(R) \to H^i_{\m}(R)$ is injective. 
Then $z$ has to be zero, because $cF^e_{H^i_{\m}(R)}(z) \in \ba^{\lceil t(p^e-1) \rceil} F^e_{H^i_{\m}(R)}(z)=0$. 
That is, $0^{F \sharp \ba^t}_{H^i_{\m}(R)}=0$. 

For the latter assertion, suppose that $R$ is quasi-Gorenstein. 
Then by \cite[Theorem~4.1]{Sch}, $(R, \ba^t)$ is sharply $F$-pure if and only if for infinitely many $q=p^e$, there exists $c \in \ba^{\lceil t(q-1) \rceil}$ such that $cF^e_{H^d_{\m}(R)} \colon H^d_{\m}(R) \to H^d_{\m}(R)$ is injective. 
Looking at the socle of $H^d_{\m}(R)$, we see that this condition is equivalent to saying that $0^{F \sharp \ba^{t}}_{H^d_{\m}(R)}=0$. 

(4) Since $H^{d}_{\n}(S)$ does not change by passing to the completion of $S$, we may assume that $S$ is complete. 
Let $\mathrm{Tr}_R^e \colon F^e_*\omega_R \to \omega_R$ (resp. $\mathrm{Tr}_S^e \colon F^e_*\omega_S \to \omega_S$) denote the $e$-th iteration of the trace map on $R$ (resp. $S$). 
It then follows from the proof of \cite[Lemma~1.5~(2)]{SrT} that $\mathrm{Tr}_R^e \otimes_R S \colon F^e_*\omega_R \otimes_R S \to \omega_R \otimes_R S$ is isomorphic to $\mathrm{Tr}_S^e$ for each $e \in \N$. 
By Remark~\ref{remark:hara-takagi}, $0^{F \sharp \ba^t}_{H^d_{\m}(R)}=0$ if and only if 
\[
\bigoplus_{e \ge e_0}\mathrm{Tr}_R^e \colon \bigoplus_{e \ge e_0} F^e_*(\ba^{\lceil t (p^e-1) \rceil}\omega_R) \to \omega_R
\]
is surjective for every integer $e_0 \ge 0$. 
Tensoring with $S$, we see that this condition is equivalent to the surjectivity of 
\[
\bigoplus_{e \ge 0} \mathrm{Tr}_S^e \colon \bigoplus_{e \ge e_0} F^e_*(\ba^{\lceil t(p^e-1) \rceil}\omega_S) \to \omega_S
\]
for every $e_0 \ge 0$, which holds by Remark~\ref{remark:hara-takagi} again if and only if $0^{F \sharp (\ba S)^t}_{H^d_{\n}(S)}=0$. 
\end{proof}

\section{Positive characteristic case I}
\label{sec:char p I}

We introduce a new invariant of singularities in positive characteristic, and study its basic properties. 
Using this, we give a partial answer to Conjecture~\ref{conj:HWY}. 

\begin{defn}
\label{def:thresholds}
Let $(R, \m)$ be a $d$-dimensional $F$-finite $F$-injective local ring of characteristic $p>0$, $\ba$ be an ideal of $R$ such that $\ba \cap R^{\circ} \ne \emptyset$. For each integer~$i$, the threshold $c_i(\ba)$ is defined by 
\[
c_i(\ba)=\sup\{t \in \R_{\ge 0} \bigm| 0^{F\sharp \ba^t}_{H^i_{\m}(R)}=0\}.
\]
Note that $c_i(\ba)=\infty$ when $H^i_{\m}(R)=0$. 
Also, we simply denote $c_d(\ba)$ by $c(\ba)$. 
\end{defn}

\begin{rem}
\label{rem:graded case}
Let $R$ be an $\N$-graded ring with $R_0$ an $F$-finite field of characteristic $p>0$, and~$\m$ the homogeneous maximal ideal of $R$. 
Then we can define $c_i(\m)$ similarly, that is, 
\[
c_i(\m)=\sup\{t \in \R_{\ge 0} \bigm| 0^{F\sharp \m^t}_{H^i_{\m}(R)}=0\}.
\]
Since $H^i_{\m}(R) \cong H^i_{\m R_{\m}}(R_{\m})$, we have the equality $c_i(\m)=c_i(\m R_{\m})$. 
\end{rem}

\begin{lem}
\label{lem:basic}
Let the notation be the same as in Definition~\ref{def:thresholds}. 
\begin{enumerate}[\quad\rm(1)]
\item If $R$ is Cohen-Macaulay, then 
\[
c(\ba)=\sup\{t \in \R_{\ge 0} \mid \textup{$(R, \ba^t)$ is sharply $F$-injective}\}.
\]

\item If $R$ is $F$-rational, then 
\[
c(\ba)=\sup\{t \in \R_{\ge 0} \mid \textup{$(R, \ba^t)$ is $F$-rational}\}.
\]

\item Suppose that $R$ is $F$-pure. Then $c_i(\ba) \ge \mathrm{fpt}(\ba)$ for each $i$. 
In addition, if $R$ is quasi-Gorenstein, then $c(\ba)=\mathrm{fpt}(\ba)$. 

\item $c(\ba)$ is less than or equal to the height $\mathrm{ht}\; \ba$ of $\ba$. 

\item Suppose that $R$ is Cohen-Macaulay and the residue field $R/\m$ is infinite. 
If $J \subset R$ is a minimal reduction of $\m$, then $\m^{d+1-\lceil c(\m) \rceil} \subseteq J$. 

\item Suppose that $R$ is Cohen-Macaulay. 
If $c(\m)>d-1$, then $R$ is regular and in particular $c(\m)=d$. 
\end{enumerate}
\end{lem}

\begin{proof}
(1) (resp. (2), (3)) follows from Lemma~\ref{lemma:basic pair} (1) (resp. (2), (3)).

(4) Since the trace map commutes with localization, by Remark~\ref{remark:hara-takagi}, $c(\ba) \le c(\ba R_{\p})$ for every prime ideal $\p$ containing $\ba$. 
Localizing at a minimal prime of $\ba$, we may assume that $\mathrm{ht}\; \ba=d$. 
We can also assume by Lemma~\ref{lemma:basic pair} (4) that the residue field $R/\m$ is infinite. 
Let $J$ be a minimal reduction of $\ba$, and we will show that $0^{F\sharp \ba^t}_{H^d_{\m}(R)} \supseteq (0:J)_{H^d_{\m}(R)} \ne 0$ for every $t>d$. 
Let $z \in (0:J)_{H^d_{\m}(R)}$. 
Note that if $t>d$, then $J^{[q]} \supseteq \ba^{\lceil t(q-1) \rceil}$ for all large $q=p^e$ (because $\ba^{\lceil t(q-1) \rceil}=J^{\lceil t'_q(q-1) \rceil}\ba^n$ for some fixed $n\in\mathbb{N}$, where $t'_q=t-n/(q-1)$). 
Then 
\[
\ba^{\lceil t(q-1) \rceil} F^e_{H^d_{\m}(R)}(z) \subseteq J^{[q]}F^e_{H^d_{\m}(R)}(z)=F^e_{H^d_{\m}(R)}(Jz)=0
\]
for such $q=p^e$, which implies that $z \in 0^{F\sharp \ba^t}_{H^d_{\m}(R)}$. 

(5) It follows from Remark~\ref{remark:hara-takagi} that $\sum_{e \ge e_0}\mathrm{Tr}^e(F^e_*(\m^{\lceil (c(\m)-\epsilon)(p^e-1) \rceil}\omega_R))=\omega_R$ for all $e_0 \in \Z_{\ge 0}$ and for all $c(\m) \ge \epsilon >0$ (when $c(\m)=0$, we put $\epsilon=0$). 
Multiplying by $\m^{d+1-\lceil c(\m) \rceil}$ on both sides, one has 
\begin{align*}
\m^{d+1-\lceil c(\m) \rceil}\omega_R & =\m^{d+1-\lceil c(\m) \rceil} \sum_{e \ge e_0}\mathrm{Tr}^e\left(F^e_*(\m^{\lceil (c(\m)-\epsilon)(p^e-1) \rceil}\omega_R)\right)\\
& \subseteq \sum_{e \ge e_0}\mathrm{Tr}^e\left(F^e_*(\m^{\lceil (c(\m)-\epsilon)(p^e-1) \rceil+(d+1-\lceil c(\m) \rceil)p^e}\omega_R)\right)\\
& \subseteq \sum_{e \ge e_0}\mathrm{Tr}^e\left(F^e_*(\m^{d p^e}\omega_R)\right)\\
& \subseteq \sum_{e \ge e_0}\mathrm{Tr}^e\left(F^e_*(J^{[p^e]}\omega_R)\right)\\
& \subseteq J \omega_R
\end{align*}
for sufficiently large $e_0$ and for sufficiently small $\epsilon>0$. 
Since $R/J$ is the Matlis dual of $\omega_R/J \omega_R$, this means that $\m^{d+1-\lceil c(\m) \rceil} \subseteq J$. 

(6) Let $J$ be a minimal reduction of $\m$; we may assume by Lemma~\ref{lemma:basic pair} (4) that the residue field $R/\m$ is infinite. 
It then follows from (5) that $\m=J$, which means that $\m$ is generated by at most $d$ elements, 
that is, $R$ is regular. 
If $R$ is regular, then $c(\m)=\mathrm{fpt}(\m)=d$ by (3) and \cite[Theorem~2.7~(1)]{TW}. 
\end{proof}

\begin{eg}
Let $S$ be the $n$-dimensional polynomial ring $k[x_1, \dots, x_n]$ over an $F$-finite field $k$. 
Let $R=S^{(r)}$ be the $r$-th Veronese subring of $S$ and $\m_R$ be the homogeneous maximal ideal of $R$. 
Then $\mathrm{fpt}(\m_R)=n/r$ and $c(\m_R)=\lceil n/r \rceil$. 
\end{eg}

When $R$ is an $\N$-graded ring, the $i$-th $a$-invariant $a_i(R)$ is defined by
\[
a_i(R)=\max\{n \in \Z \mid [H^i_{\m}(R)]_n \ne 0\}
\]
for each $i$. 
The following proposition can be viewed as an extension of \cite[Theorem~4.3]{dSNB}. 

\begin{prop}
\label{prop:a_i}
Let $R$ be an $F$-injective $\N$-graded ring, with $R_0$ an $F$-finite field of characteristic $p>0$. 
Let $\m$ be the homogeneous maximal ideal of $R$. 
Then $c_i(\m) \le -a_i(R)$ for each $i$. In particular, if $R$ is $F$-pure, then by Lemma~\ref{lem:basic}~(3), one has the inequality $\mathrm{fpt}(\m) \le -a_i(R)$ for every integer $i$. 
\end{prop}

\begin{proof}
We may assume that $H^i_{\m}(R) \ne 0$. 
We will then show that $c_i(\m) \le -a_i(R)+\epsilon$, that is, $0^{F\sharp \m^{-a_i(R)+\epsilon}}_{H^i_{\m}(R)} \ne 0$, for every $\epsilon>0$. 
Note that $a_i(R) \le 0$, because $R$ is $F$-injective. 
Let $z \in [H^i_{\m}(R)]_{a_i(R)}$ be a nonzero element. 
Since $\lceil (-a_i(R)+\epsilon)(q-1) \rceil +a_i(R)q >0$ for all sufficiently large $q=p^e$, 
one has 
\[
\m^{\lceil (-a_i(R)+\epsilon)(q-1) \rceil}F^e_{H^i_{\m}(R)}(z) \subseteq [H^i_{\m}(R)]_{>0}=0
\]
for such $q$, 
which means that $z \in 0^{F\sharp \m^{-a_i(R)+\epsilon}}_{H^i_{\m}(R)}$. 
\end{proof}

We record two cases in which the above result can be strengthened to an equality:

\begin{prop}
\label{prop:fit=-a}
Let $R$ be an $F$-injective standard graded ring with $R_0$ an $F$-finite field of characteristic $p>0$. 
Let $\m$ denote the homogeneous maximal ideal of $R$. 
Suppose that one of the following conditions is satisfied: 
\begin{enumerate}[\quad\rm(1)]
\item $R$ is Cohen-Macaulay, 
\item $R$ is normal and quasi-Gorenstein. 
\end{enumerate}
Then $c(\m)=-a(R)$. 
\end{prop}

\begin{proof}
Let $d$ be the dimension of $R$. 
First, we assume that the condition (1) holds and we will prove that $c(\m) \ge -a(R)-\epsilon$ for every $\epsilon>0$. 
We may assume by Lemma~\ref{lemma:basic pair} (4) that $R_0$ is an infinite field. 
Let $J$ be a minimal reduction of $\m$. 
As $R$ is a Cohen-Macaulay standard graded ring, $J$ is generated by a homogeneous regular sequence of degree one. 
Then $\m^{d+a(R)+1}$ is contained in $J$ but $\m^{d+a(R)}$ is not. 

It is enough to show that $J^{F\sharp \m^{-a(R)-\epsilon}}=J$ by Lemma~\ref{lemma:basic pair} (1) and Lemma~\ref{lem:basic} (1). 
Let $x \in J^{F\sharp \m^{-a(R)-\epsilon}}$, and we may assume that the degree of $x$ is less than or equal to $d+a(R)$. 
By definition, $\m^{\lceil (-a(R)-\epsilon)(q-1) \rceil}x^q \subseteq J^{[q]}$ for all sufficiently large $q=p^e$. 
Thus, 
\[
x^q \in (J^{[q]}:J^{\lceil (-a(R)-\epsilon)(q-1) \rceil}) \subseteq J^{[q]}+J^{dq-\lceil (-a(R)-\epsilon)(q-1) \rceil -d+1}.
\]
The degree of $x^q$ is less than or equal to $(d+a(R))q$, but $dq-\lceil (-a(R)-\epsilon)(q-1) \rceil -d+1$ is greater than $(d+a(R))q$ for sufficiently large $q$, so $x^q$ has to lie in $J^{[q]}$ for such $q$. 
It then follows from the $F$-injectivity of $R$ that $x \in J$, that is, $J^{F\sharp \m^{-a(R)-\epsilon}}=J$. 

Next, we assume that the condition (2) holds and we will show that $c(\m) \ge -a(R)$. 
It is enough to show by Lemma~\ref{lem:basic} (3) that $\mathrm{fpt}(\m) \ge -a(R)$. 
Let $X=\Proj R$. 
Since $R$ is a quasi-Gorenstein normal standard graded ring, there exists a very ample divisor $H$ on $X$ such that $R=\bigoplus_{n \ge 0}H^0(X, \sO_X(nH))$ and $K_X \sim a(R)H$. 
Note that $X$ is globally $F$-split and $a(R) \le 0$, because $R$ is $F$-pure. 
It then follows from an argument similar to the proof of \cite[Theorem~4.3]{SS} that there exists an effective Cartier divisor $D$ on $X$ such that $D \sim (1-p)a(R)H$ and that the composite map 
\[
\sO_X \to F_*\sO_X \to F_*\sO_X(D) \quad x \mapsto F_*x^p \mapsto F_*(sx^p)
\]
splits as an $\sO_X$-module homomorphism, where $s$ is a defining section for $D$. 
This map induces the $R$-linear map $R \to F_*R$ with $1\mapsto F_*s$, which also splits. 
Since $s$ belongs to $\m^{(1-p)a(R)}$ by the definition of $D$, the pair $(R, \m^{-a(R)})$ is sharply $F$-pure. Thus, $\mathrm{fpt}(\m) \ge -a(R)$. 
\end{proof}

Motivated by Conjecture~\ref{conj:HWY}, we propose the following conjecture. 

\begin{conj}
\label{conj:fpt=fit}
Let $(R, \m)$ be an $F$-finite $F$-pure normal local ring of characteristic $p>0$. Then $R$ is quasi-Gorenstein if and only if $\mathrm{fpt}(\m)=c(\m)$. 
\end{conj}

\begin{rem}
Conjecture~\ref{conj:fpt=fit} can fail if $R$ is not normal. 
Indeed, \cite[Example~5.3]{dSNB} and Proposition~\ref{prop:fit=-a} give a counterexample. 
\end{rem}

Conjecture~\ref{conj:fpt=fit} implies an extension of Conjecture~\ref{conj:HWY}.

\begin{prop}
\label{prop:imply}
Let $R$ be an $F$-pure normal standard graded ring, with $R_0$ an $F$-finite field of characteristic $p>0$. 
Let $\m$ denote the homogeneous maximal ideal of $R$. 
Suppose that Conjecture~\ref{conj:fpt=fit} holds for the localization $R_{\m}$ of $R$ at $\m$. 
Then $R$ is quasi-Gorenstein if and only if $\mathrm{fpt}(\m)=-a(R)$. 
\end{prop}

\begin{proof}
The ``only if" part immediately follows from Lemma~\ref{lem:basic} (3) and Proposition~\ref{prop:fit=-a}. We will show the ``if" part. 
Suppose that $\mathrm{fpt}(\m)=-a(R)$. Then by Remark~\ref{rem:graded case}, Lemma~\ref{lem:basic}~(3) and Proposition~\ref{prop:fit=-a}, 
\[
-a(R)=\mathrm{fpt}(\m) \le \mathrm{fpt}(\m R_{\m}) \le c(\m R_{\m})=c(\m)=-a(R),
\]
which implies that $\mathrm{fpt}(\m R_{\m}) = c(\m R_{\m})$. 
It then follows from Conjecture~\ref{conj:fpt=fit} that $R_{\m}$ is quasi-Gorenstein, which is equivalent to saying that $R$ is quasi-Gorenstein. 
\end{proof}

\begin{thm}
\label{thm:non-Gorenstein}
Let $(R, \m)$ be an $F$-finite normal local ring of characteristic $p>0$ and $\Delta$ be an effective $\Q$-divisor on $X:=\Spec R$ such that $(X, \Delta)$ is $F$-pure and $K_X+\Delta$ is~$\Q$-Cartier of index $r$. 
If $R$ is not quasi-Gorenstein, then 
\[
\mathrm{fpt}(\Delta; \m)+\frac{1}{r} \le c(\m).
\]
\end{thm}

\begin{proof}
Let $d$ be the dimension of $R$. 
For every $\mathrm{fpt}(\Delta; \m)>\epsilon>0$ (when $\mathrm{fpt}(\Delta; \m)=0$, put $\epsilon=0$), 
by the definition of $\mathrm{fpt}(\Delta; \m)$ and Lemma~\ref{lem:sharp F-purity}, there exist $q_0=p^{e_0}$ and $c$ in~$\m^{\lfloor (\mathrm{fpt}(\Delta; \m)-\epsilon)(q_0-1) \rfloor}$ such that 
\[
cF^{e_0}_{X, \Delta} \colon H^d_{\m}(\omega_X) \to H^d_{\m}(\sO_X(\lfloor q_0K_X+(q_0-1)\Delta \rfloor))
\]
is injective. 
We consider the following commutative diagram:
\[
\xymatrix{
\omega_X \times H^d_{\m}(R) \ar[r] \ar[d] & H^d_{\m}(\omega_X) \ar[d]^{cF^{e_0}_{X, \Delta}} \\
\sO_X(\lfloor q_0K_X+ (q_0-1)\Delta \rfloor) \times H^d_{\m}(R) \ar[r] & H^d_{\m}(\sO_X(\lfloor q_0K_X+ (q_0-1)\Delta \rfloor)),
}
\]
where the left vertical map sends $(x, z)$ to $(cx^{q_0}, F^{e_0}_{H^d_{\m}(R)}(z))$. 

For each $1/r>\epsilon'>0$, we will show that $0_{H^d_{\m}(R)}^{F\sharp \m^{\mathrm{fpt}(\Delta; \m)+1/r-\epsilon-\epsilon'}}=0$, which implies the assertion. 
Let $\xi \in 0_{H^d_{\m}(R)}^{F\sharp \m^{\mathrm{fpt}(\Delta; \m)+1/r-\epsilon-\epsilon'}}$, that is, there exists $q_1 \in \N$ such that 
\[
\m^{\lceil (\mathrm{fpt}(\Delta; \m)+1/r-\epsilon-\epsilon')(q-1) \rceil}F^e_{H^d_{\m}(R)}(\xi)=0
\]
for all $q=p^e \ge q_1$. 
By the definition of weak $F$-purity, we may assume that $q_0$ is sufficiently large so that $q_0 \ge q_1$ and $\epsilon'(q_0-1) \ge 2$. 
Since $R$ is not quasi-Gorenstein, $r \ge 2$ or $\Delta$ is strictly effective. 
In either case, the fractional ideal $\omega_X^{r}=\sO_X(K_X)^{r}$ is contained in the fractional ideal $\m \sO_X(r(K_X+\Delta))$. 
Therefore, for all $x \in \omega_X$, one has 
\begin{align*}
cx^{q_0} & \in \m^{\lfloor (\mathrm{fpt}(\Delta; \m)-\epsilon)(q_0-1) \rfloor} \omega_X^{q_0}\\ 
& \subseteq \m^{\lfloor (\mathrm{fpt}(\Delta; \m)-\epsilon)(q_0-1) \rfloor+\lfloor (q_0-1)/r \rfloor} \sO_X(\lfloor q_0K_X+(q_0-1)\Delta \rfloor)\\
& \subseteq \m^{\lceil (\mathrm{fpt}(\Delta; \m)+1/r-\epsilon-\epsilon')(q_0-1) \rceil} \sO_X(\lfloor q_0K_X+(q_0-1)\Delta \rfloor) 
\end{align*}
by the choice of $q_0$. 
Since $q_0 \ge q_1$, 
\begin{align*}
cx^{q_0} F^{e_0}_{H^d_{\m}(R)}(\xi) &\in \m^{\lceil (\mathrm{fpt}(\Delta; \m)+1/r-\epsilon-\epsilon')(q_0-1) \rceil} \sO_X(\lfloor q_0K_X+(q_0-1)\Delta \rfloor) F^{e_0}_{H^d_{\m}(R)}(\xi)\\
&=0 \; \, \textup{in} \; \, H^d_{\m}(\sO_X(\lfloor q_0K_X+ (q_0-1)\Delta \rfloor)),
\end{align*}
and it then follows from the commutativity of the above diagram that $cF^{e_0}_{X, \Delta}(x\xi)=0$. 
The injectivity of the map $cF^{e_0}_{X, \Delta}$ implies that $x \xi=0$ for all $x \in \omega_X$. 
This forces $\xi$ to be zero, because $\omega_X \times H^d_{\m}(R) \to H^d_{\m}(\omega_X)$ is the duality pairing. 
Thus, $0_{H^d_{\m}(R)}^{F\sharp \m^{\mathrm{fpt}(\Delta; \m)+1/r-\epsilon-\epsilon'}}=0$. 
\end{proof}

We give an example of a standard graded Cohen-Macaulay ring $R$ that is $F$-pure, with $a(R)=-1$ and $\mathrm{fpt}(\m)=0$; this is based on \cite{Song}. The ring $R$ is $\Q$-Gorenstein, with index $2$.

\begin{eg}
Let $k$ be a field of characteristic $p\equiv 1\mod 4$, and set
\[
S=k[w,x,y,z]/(w^4+x^4+y^4+z^4)\,.
\]
By the characteristic assumption, the ring $S$ is $F$-pure. Set $R$ to be the $k$-subalgebra generated by the monomials
\[
w^4,\ w^3x,\ w^2x^2,\ wx^3,\ x^4,\quad y^4,\ y^3z,\ y^2z^2,\ yz^3,\ z^4.
\]
Then $R$ is a direct summand of $S$ as an $R$-module: one way to see this is to use the $\Z/4\times\Z/4$-grading on $S$ under which $\deg w=(1,0)=\deg x$, and $\deg y=(0,1)=\deg z$, in which case~$R$ is the subring of $S$ generated by elements of degree $(0,0)$. It follows that~$R$ is $F$-pure, normal, as well as Cohen-Macaulay.

The ring $R$ has a standard grading under which each of the monomials displayed is assigned degree one. Computing the socle modulo the system of parameters $x^4$, $y^4$, $z^4$, it follows that $a(R)=-1$. By Proposition~\ref{prop:fit=-a}, we have $c(\m)=1$.

The fractional ideal
\[
\omega_R = \frac{1}{w^2x^2}(w^3x, w^2x^2, wx^3)(y^3z, y^2z^2, yz^3)
\]
is, up to isomorphism, the graded canonical module of $R$; its second symbolic power is
\[
\omega_R^{(2)} = \frac{y^2z^2}{w^2x^2}R,
\]
so the ring $R$ is $\Q$-Gorenstein. Using Theorem~\ref{theorem:omega}, one checks that that $\nu_e(\m)=0$ for each $e\ge 1$. It follows that $\mathrm{fpt}(\m)=0$.
\end{eg}

\begin{cor}
\label{cor:Delta}
Let $(R, \m)$ be an $F$-finite $F$-pure normal local ring of characteristic $p>0$. 
\begin{enumerate}[\quad\rm(1)]
\item Suppose that there exists an effective $\Q$-divisor $\Delta$ on $X=\Spec R$ such that $K_X+\Delta$ is $\Q$-Cartier, $(X, \Delta)$ is $F$-pure and $\mathrm{fpt}(\Delta; \m)=\mathrm{fpt}(\m)$. 
Then Conjecture~\ref{conj:fpt=fit} holds for this $R$. 

\item If $c(\m)=0$, then $R$ is quasi-Gorenstein. 
\end{enumerate}
\end{cor}

\begin{proof}
(1) immediately follows from Theorem~\ref{thm:non-Gorenstein}. 
We will show (2). 
Since $R$ is $F$-pure, then by \cite[Theorem~4.3~(ii)]{SS}, there exists an effective $\Q$-divisor $\Delta$ on $X$ such that $(R, \Delta)$ is sharply $F$-pure with $K_X+\Delta$ $\Q$-Cartier. Then 
\[
0 \le \mathrm{fpt}(\Delta;\m) \le \mathrm{fpt}(\m) \le c(\m)=0,
\]
and the assertion follows from (1). 
\end{proof}

When is the assumption of Corollary~\ref{cor:Delta}~(1) satisfied? 
If the pair $(R, \m^{\mathrm{fpt}(\m)})$ is sharply $F$-pure, then by a similar argument to the proof of \cite[Theorem~4.3~(ii)]{SS}, there exists an effective $\Q$-divisor $\Delta$ on $X$ such that $((R, \Delta);\m^{\mathrm{fpt}(\m)})$ is sharply $F$-pure with $K_X+\Delta$ $\Q$-Cartier. 
Then $\mathrm{fpt}(\Delta; \m)=\mathrm{fpt}(\m)$, that is, the assumption of Corollary~\ref{cor:Delta}~(1) is satisfied. 

\begin{ques}[{cf.~\cite[Question~3.6]{HWY}}]
Let $(R, \m)$ be an $F$-finite $F$-pure normal local ring of characteristic $p>0$. 
When is the pair $(R, \m^{\mathrm{fpt}(\m)})$ sharply or weakly $F$-pure? 
\end{ques}

We will show in Proposition~\ref{prop:critical value} that if $(R, \m)$ is an $F$-pure $\Q$-Gorenstein normal standard graded ring over an $F$-finite field of characteristic $p>0$ with Gorenstein index not divisible by $p$, 
then $(R, \m^{\mathrm{fpt}(\m)})$ is sharply $F$-pure. 

We now prove the main result of this section:

\begin{thm}
\label{thm:anti-canonical}
Let $(R, \m)$ be an $F$-finite $F$-pure normal local ring of characteristic $p>0$. 
Suppose that the anti-canonical cover $\bigoplus_{n \ge 0}\sO_X(-nK_X)$ of $X:=\Spec R$ is Noetherian. 
Then $\mathrm{fpt}(\m)=c(\m)$ if and only if $R$ is quasi-Gorenstein. 
\end{thm}

\begin{proof}
Since $\bigoplus_{n \ge 0}\sO_X(-nK_X)$ is Noetherian, one can find an integer $r \ge 1$ satisfying the following: for every $q=p^e$, if we write $q-1=rn_e+j_e$ with $n_e \ge 0$ and $r-1 \ge j_e \ge 0$, then 
\[
\sO_X((1-q)K_X)=\sO_X(-rK_X)^{n_e} \sO_X(-j_eK_X).
\]
Suppose that $R$ is not quasi-Gorenstein, and we will show that $\mathrm{fpt}(\m)+\frac{1}{r} \le c(\m)$. 
Let $\phi_1, \dots, \phi_l$ be a system of generators for $\sO_X(-rK_X)$. 

For every $\mathrm{fpt}(\m) >\epsilon>0$, there exist a sufficiently large $q=p^e$ and $c \in \m^{\lfloor (\mathrm{fpt}(\m)-\epsilon)(q-1) \rfloor}$ such that the $R$-linear map $R \to F^e_*R$ sending $1$ to $F^e_*c$ splits (when $\mathrm{fpt}(\m)=0$, put $\epsilon=0$ and $c=1$). That is, there exists an $R$-linear map $\phi \colon F^e_*R \to R$ sending $F^e_*c$ to $1$. 
It follows from Grothendieck duality that there exists an isomorphism 
\[
\Phi \colon \mathrm{Hom}_R(F^e_*R, R) \cong F^e_*\sO_X((1-q)K_X)=F^e_*\left(\sO_X(-rK_X)^{n_e} \sO_X(-j_eK_X)\right).
\]
We write $\Phi(\phi)=\sum_{\underline{m}} F^e_*(\phi_1^{m_1} \cdots \phi_l^{m_l} \psi_{\underline{m}})$ with $\psi_{\underline{m}} \in \sO_X(-j_eK_X)$, where $\underline{m}$ runs through all elements of $\{(m_1, \dots, m_l) \in \Z_{\ge 0}^l \mid m_1+\cdots+m_l=n_e\}$. 
Then 
\begin{align*}
1=\phi(F^e_*c)
&=\Phi^{-1}\left(\sum_{\underline{m}} F^e_*(\phi_1^{m_1} \cdots \phi_l^{m_l} \psi_{\underline{m}})\right)(F^e_*c)\\
&=\sum_{\underline{m}}\Phi^{-1}\left(F^e_*(\phi_1^{m_1} \cdots \phi_l^{m_l} \psi_{\underline{m}})\right)(F^e_*c). 
\end{align*}
Therefore, there exists $\underline{m}=(m_1, \dots, m_l) \in \Z_{\ge 0}^l$ with $\sum_{i=1}^l m_i=n_e$ such that \linebreak $\Phi^{-1}(F^e_*(\phi_1^{m_1} \cdots \phi_l^{m_l} \psi_{\underline{m}}))(F^e_*c)$ is a unit. 
Replacing $c$ by a unit multiple, we may assume that $\Phi(\phi)=F^e_*(\phi_1^{m_1} \cdots \phi_l^{m_l} \psi_{\underline{m}})$. 

Since each $\phi_i$ determines an effective divisor $D_i$ which is linearly equivalent to $-rK_X$, the section $F^e_*(\phi_1^{m_1} \cdots \phi_l^{m_l} \psi_{\underline{m}})$ lies in $F^e_*\sO_X((1-q)K_X-m_1D_1-\cdots-m_lD_l)$ and we have the following commutative diagram: 
\[
\xymatrix{
F^e_*\sO_X((1-q)K_X-m_1D_1-\cdots-m_lD_l) \ar[r] & F^e_*\sO_X((1-q)K_X)\\
\mathrm{Hom}_R(F^e_*\sO_X(m_1D_1+\cdots+m_lD_l), R) \ar[r] \ar[u] & \mathrm{Hom}_R(F^e_*R, R) \ar[u]_{\Phi}, 
}
\]
where the vertical maps are isomorphisms. 
Therefore, $\phi$ induces an $R$-linear map 
\[F^e_*\sO_X(m_1D_1+\cdots+m_lD_l) \to R\] sending $F^e_*c$ to $1$. 
Then its Matlis dual 
\[
cF^e \colon H^d_{\m}(\omega_R) \to H^d_{\m}(\sO_X(qK_X+m_1D_1+\cdots+m_lD_l))
\]
is injective. 
On the other hand, since $R$ is not quasi-Gorenstein and $rK_X+D_i\sim 0$, the fractional ideal $\omega_X^r=\sO_X(K_X)^r$ is contained in $\m \sO_X(rK_X+D_i)$ for each $i=1, \dots, l$. 
Hence, $\omega_X^q$ is contained in $\m^{\lfloor (q-1)/r \rfloor} \sO_X(qK_X+m_1D_1+\cdots+m_lD_l)$. 
It then follows from an analogous argument to the proof of Theorem~\ref{thm:non-Gorenstein} that $\mathrm{fpt}(\m)+\frac{1}{r} \le c(\m)$. 
\end{proof}

\begin{rem}
\label{rem:finite generation}
In the setting of Theorem~\ref{thm:anti-canonical}, it is well-known that $\bigoplus_{n \ge 0}\sO_X(-nK_X)$ is Noetherian if $R$ is $\Q$-Gorenstein, $R$ is a normal semigroup ring or $R$ is a determinantal ring. 
We briefly explain the reason why $\bigoplus_{n \ge 0}\sO_X(-nK_X)$ is Noetherian in the latter case. 

Let $R$ be the determinantal ring $k[T]/I$, where $T$ is an $m\times n$ matrix of indeterminates with $m\le n$, and $I$ is the ideal generated by the size $t$ minors of~$T$ where $1\le t\le m$. Then the anti-canonical class of $R$ is the class of the $(n-m)$-th symbolic power of the prime ideal~$\p$ generated by the size $t-1$ minors of the first $t-1$ rows of $T$ by \cite[Theorem~8.8]{BV}. Moreover, the symbolic powers of $\p$ coincide with its ordinary powers by \cite[Corollary~7.10]{BV}, so the anti-canonical cover is the Rees algebra of $\p^{n-m}$. In particular, it is Noetherian.
\end{rem}

Another case where $\bigoplus_{n \ge 0}\sO_X(-nK_X)$ is Noetherian is the following: 

\begin{cor}
\label{cor:dim 3}
Let $X$ be a three-dimensional strongly $F$-regular variety over an algebraically closed field of characteristic $p>5$ and $x$ be a closed point of $X$. 
Then $\mathrm{fpt}(\m_x)=c(\m_x)$ if and only if $X$ is Gorenstein at $x$. 
\end{cor}

\begin{proof}
We may assume that $X$ is affine. 
By \cite[Theorem~4.3]{SS} and \cite[Theorem~3.3]{HW}, there exists an effective $\Q$-divisor $\Delta$ on $X$ such that $K_X+\Delta$ is $\Q$-Cartier and $(X, \Delta)$ is strongly $F$-regular and in particular is klt. 
Since the minimal model program holds for three-dimensional klt pairs in characteristic $p>5$, the anti-canonical cover $\bigoplus_{n \ge 0}\sO_X(-nK_X)$ is Noetherian (see for example \cite[Theorem~2.28]{CEMS}). Thus, the assertion follows from Theorem~\ref{thm:anti-canonical}. 
\end{proof}

A combination of Proposition~\ref{prop:imply}, Theorem~\ref{thm:anti-canonical}, Remark~\ref{rem:finite generation} and Corollary~\ref{cor:dim 3} gives an extension of \cite[Theorem~1.2~(2)]{HWY}:

\begin{cor}
\label{cor:conclusion I}
Let $R$ be an $F$-pure normal standard graded ring, with $R_0$ an $F$-finite field of characteristic $p>0$. Let $\m$ be the homogeneous maximal ideal of $R$. 
Suppose that the anti-canonical cover $\bigoplus_{n \ge 0} \sO_X(-nK_X)$ of $X:=\Spec R$ is Noetherian. This assumption is satisfied, for example, in each of the following cases: 
\begin{enumerate}[\quad\rm(1)]
\item $R$ is $\Q$-Gorenstein, 
\item $R$ is a semigroup ring,
\item $R$ is a determinantal ring, 
\item $R$ is a strongly $F$-regular ring of dimension at most three,
\item $R$ is a four-dimensional strongly $F$-regular ring and $p>5$. 
\end{enumerate}
Then $\mathrm{fpt}(\m)=-a(R)$ if and only if $R$ is quasi-Gorenstein. 
\end{cor}

\begin{proof}
If $\bigoplus_{n \ge 0} \sO_X(-nK_X)$ is Noetherian, then the assertion follows from Proposition~\ref{prop:imply} and Theorem~\ref{thm:anti-canonical}. 
By Remark~\ref{rem:finite generation}, $\bigoplus_{n \ge 0} \sO_X(-nK_X)$ is Noetherian in the case of (1), (2) and (3). 
We will explain why $\bigoplus_{n \ge 0} \sO_X(-nK_X)$ is Noetherian in the case of (4) and (5). 

Since two-dimensional strongly $F$-regular rings are $\Q$-Gorenstein, we may assume that $\dim X \ge 3$. 
Also, since strong $F$-regularity is preserved under flat base change by \cite[Theorem~3.6]{Ab}, we may assume that $R_0$ is an algebraically closed field. 
Let $D$ be a very ample divisor on $Y:=\Proj R$ so that $R=\bigoplus_{m \ge 0}H^0(Y, \sO_Y(mD))$. 
It follows from \cite[Theorem~1.1]{SS} that $Y$ is a normal projective variety of Fano type of dimension at most three. 
It is known that the minimal model program holds for klt surfaces and also for three-dimensional klt pairs in characteristic $p>5$ (see \cite{HX, Bi, BW}). 
Thus, applying essentially the same argument as the proof of \cite[Corollary~1.1.9]{BCHM}, we can see that 
\[
\bigoplus_{n \ge 0}\sO_X(-nK_X) \cong \bigoplus_{m \in \Z} \bigoplus_{n \ge 0}H^0(Y, \sO_Y(mD-nK_Y))
\]
is Noetherian. 
\end{proof}

We also give an answer to \cite[Question~6.7]{HWY}. 
Before stating the result, we fix some notation. 
Let $M=\Z^d$, $N=\Hom_{\Z}(M, \Z)$, and denote the duality pairing between $M_{\R}:=M\otimes_{\Z} \R$ and $N_{\R}:=N\otimes_{\Z} \R$ by $\langle - , - \rangle\colon M_{\R} \times N_{\R} \to \R$. 
Let $\sigma \subset N_{\R}$ be a strongly convex rational polyhedral cone and denote its dual cone by $\sigma^{\vee}$. 
Let $R=k[\sigma^{\vee} \cap M]$ be the affine semigroup ring over a field $k$ defined by $\sigma$ 
and $\m$ be the unique monomial maximal ideal of $R$. 
The Newton Polyhedron $P(\m) \subseteq M_{\R}$ of $\m$ is defined as the convex hull of the set of exponents $m \in M$ of monomials $x^m \in \m$. 
We define the function $\lambda_{\m}$ by 
\[
\lambda_{\m}\colon\sigma^{\vee} \to \R \quad u \mapsto \sup\{\lambda \in \R_{ \ge 0} \mid u \in \lambda P(\m)\},
\]
where we set $\lambda P(\m)=\sigma^{\vee}$ if $\lambda=0$, 
and denote 
\[
a_{\sigma}(R):=-\min\{\lambda_{\m}(u) \mid u \in \mathrm{Int}(\sigma^{\vee}) \cap M\}.
\]
Note that $a_{\sigma}(R)$ coincides with the $a$-invariant $a(R)$ if $R$ is standard graded. 

\begin{cor}
\label{cor:toric charp}
We use the above notation. 
Let $R=k[\sigma^{\vee} \cap M]$ be a (not necessarily standard graded) affine semigroup ring over an $F$-finite field $k$ of characteristic $p>0$ defined by $\sigma$.
\begin{enumerate}[\quad\rm(1)]
\item Then $c(\m)=-a_{\sigma}(R)$. 
\item $\mathrm{fpt}(\m)=-a_{\sigma}(R)$ if and only if $R$ is Gorenstein. 
\end{enumerate}
\end{cor}

\begin{proof}
Since (2) follows from (1) and Theorem~\ref{thm:anti-canonical}, we will show only (1). 
Let $v_1, \dots, v_s$ be the primitive generators for $\sigma$, that is, the first lattice points on the edges of $\sigma$. 
Note that the graded canonical module $\omega_R$ consists of the monomials $x^m$ such that $\langle m, v_i \rangle \ge 1$ for all $i=1, \dots, s$. 
Hence, its $k$-dual $H^d_{\m}(R)$ is written as 
\[
H^d_{\m}(R)=\bigoplus_{m \in S} k x^m, 
\]
where $S=\{m \in M \mid \langle m, v_i \rangle \le -1 \textup{ for all $i=1, \dots, s$}\}$. 
It follows from the fact that $1$ is an $\m^t$-test element by \cite[Theorem~6.4]{HY} (see \cite[Definition~6.3]{HY} for the definition of $\m^t$-test elements) that the pair $(R, \m^t)$ is $F$-rational if and only if for each $m \in S$, one has $\m^{\lceil tp^e \rceil}x^{p^e m} \ne 0$ in $H^d_{\m}(R)$, or equivalently, 
\[
\left(p^e m+\lceil tp^e \rceil P(\m)\right) \cap S \ne \emptyset,
\]
for infinitely many $e$. 
We can rephrase this condition as saying that $-m \in \mathrm{Int}(tP(\m))$, using an argument similar to the proof of \cite[Theorem~3]{Bl}. 
By Lemma~\ref{lem:basic}~(2), 
\begin{align*}
c(\m)&=\sup\{t \in \R_{\ge 0} \mid (R, \m^t) \textup{ is $F$-rational}\}\\
&=\sup\{t \in \R_{\ge 0} \mid -m \in \mathrm{Int}(tP(\m)) \textup{ for all $m \in S$}\}\\
&=\min_{u \in -S}\lambda_{\m}(u).
\end{align*}
Since $-S=\mathrm{Int}(\sigma^{\vee}) \cap M$, one has the equality $c(\m)=-a_{\sigma}(R)$. 
\end{proof}

\section{Positive characteristic case II}
\label{sec:char p II}

In this section we give a different interpretation of the function $\nu_e(\m)$, where $\m$ is the homogeneous maximal ideal of an $F$-pure normal standard graded domain $R$ over an $F$-finite field (Theorem~\ref{theorem:omega}). Combining it with the Fedder-type criteria (Proposition~\ref{prop:Fedder criteria}), we give explicit computations of $\mathrm{fpt}(\m)$ in many situations (e.g. Propositions~\ref{prop:determinantal} and~\ref{prop:fptQGor}), eventually yielding Corollary~\ref{cor:conclusion I} as a consequence (see Corollary~\ref{cor:final}).

\begin{thm}
\label{theorem:omega}
Let $S$ be an $n$-dimensional standard graded polynomial ring over an $F$-finite field of characteristic $p>0$. 
Let $I$ be a homogeneous ideal such that $R:=S/I$ is an $F$-pure normal domain. 
Let $\omega_R$ denote the graded canonical module of $R$.
Then, for each $q=p^e$, one has a graded isomorphism
\[
\frac{I^{[q]}:_S I}{I^{[q]}} \cong \left(\omega_R(n)\right)^{(1-q)}.
\]
In particular, if $\m$ is the homogeneous maximal ideal of $R$, then $-\nu_e(\m)$ equals the degree of a minimal generator of $\omega_R^{(1-q)}$ (See Proposition~\ref{prop:Fedder criteria} for the definition of $\nu_e(\m)$). 
\end{thm}	

\begin{proof}
After taking a flat base change, we may assume that $S=k[x_1,\dots,x_n]$, where $k$ is a perfect field. It then follows that $S$ is a free $S^q$-module with basis $x_1^{i_1}\cdots x_n^{i_n}$ where $0\le i_j\le q-1$ for each $j$. Consider the homomorphism $\phi\in\Hom_{S^q}(S,S^q)$ that maps the basis element $(x_1\cdots x_n)^{q-1}$ to $1$, and every other basis element to $0$. It is readily seen that~$\phi$ generates $\Hom_{S^q}(S,S^q)$ as an $S$-module.

Let $J$ be the ideal of $S^q$ consisting of $q$-th powers of elements of $I$; note that $JS=I^{[q]}$. Then
\[
\Hom_{S^q}(S/I,S^q/J)\ \cong\ \frac{I^{[q]}:_S I}{I^{[q]}}\phi,
\]
see \cite[page~465]{Fedder}. Next, note that one has graded isomorphisms
\begin{align*}
\Hom_{S^q}(S/I,S^q/J) &\ \cong\ \Hom_{R^q}(R,R^q)\\
&\ \cong\ \Hom_{R^q}(R,\Hom_{R^q}(\omega_{R^q},\omega_{R^q}))\\
&\ \cong\ \Hom_{R^q}(R\otimes_{R^q}\omega_{R^q},\omega_{R^q})\\
&\ \cong\ \Hom_{R^q}(\omega_R^{(q)},\omega_{R^q})\\
&\ \cong\ \Hom_{R^q}(\omega_R^{(q)}\otimes_R R,\omega_{R^q})\\
&\ \cong\ \Hom_R(\omega_R^{(q)},\Hom_{R^q}(R,\omega_{R^q}))\\
&\ \cong\ \Hom_R(\omega_R^{(q)},\omega_R)\\
&\ \cong\ \omega_R^{(1-q)}.
\end{align*}
Since the homomorphism $\phi$ has degree $n-nq$, the desired isomorphism follows.

Suppose $\omega_R^{(1-q)}$ is generated in degrees $-d_1<\dots <-d_r$, then $(I^{[q]}:_SI)/I^{[q]}$ is generated in degrees $n(q-1)-d_1<\dots <n(q-1)-d_r$. Hence the least degree of a homogeneous element of $I^{[q]}:_SI$ that is not in $\m^{[q]}$ belongs to the set 
\[
\{n(q-1)-d_1,\ldots ,n(q-1)-d_r\},
\]
and $I^{[q]}:_SI\subseteq \m^{n(q-1)-d_1}$.
Since the definition of $\nu_e(\m)$ translates as
\[
\nu_e(\m) = \max\{r\in \N \mid (I^{[q]}:_S I) \not\subseteq \m^{[q]}+\m^{n(q-1)+1-r}\},
\]
it follows that $\nu_e(\m)\in \{d_1,\ldots ,d_r\}$. 
\end{proof}

As an immediate consequence we get:

\begin{cor}
If $R$ is an $F$-pure quasi-Gorenstein standard graded normal domain, over an $F$-finite field, with homogeneous maximal ideal $\m$, then
\[
\mathrm{fpt}(\m)=-a(R).
\]
In particular, by Proposition~\ref{prop:a_i}, in this case $a(R)\ge a_i(R)$ for all $i=0,\ldots, \dim R$.
\end{cor}

\begin{prop}
\label{prop:determinantal}
Let $R$ be the determinantal ring $k[T]/I$, where $k$ is an $F$-finite field, the matrix of indeterminates $T$ has size $m\times n$ with $m\le n$, and $I$ is the ideal generated by the size $t$ minors of~$T$ where $1\le t\le m$. Let $\m$ be the homogeneous maximal ideal of $R$. Then 
\[
\mathrm{fpt}(\m) = m(t-1).
\]
\end{prop}

\begin{rem}
In the notation of the proposition, the ring $R$ has $a$-invariant $-n(t-1)$. It follows that
$\mathrm{fpt}(\m)=-a(R)$ precisely when $m=n$ or $t=1$, i.e., if and only if $R$ is Gorenstein.

In the case $t=2$, the $F$-pure threshold has been calculated previously, see \cite[Corollary~1]{CM} or \cite[Example~6.2]{HWY}. Since $I$ is a homogeneous ideal of $k[T]$, which is $F$-pure, one can also ask for $\mathrm{fpt}(I)$. This threshold has been computed in \cite{MSV} (see \cite{HV} for various generalizations):
\[
\mathrm{fpt}(I)=\min\left\{\frac{(m-l)(n-l)}{t-l} \bigm| l=0,\ldots ,t-1\right\}.
\]
\end{rem}

\begin{proof}[Proof of Proposition~\ref{prop:determinantal}]
The graded canonical module of $R$ is computed in \cite[Corollary~1.6]{BH:ainv}, namely, it equals $\q^{n-m}(m-mt)$, where $\q$ is the prime ideal generated by the size $t-1$ minors of the first $t-1$ columns of the matrix $T$. The divisor class group of $R$ is described by \cite[Corollary~7.10]{BV}, from which it follows that $\omega_R^{(1-q)}$ is generated by elements of degree $-m(q-1)(t-1)$. Theorem~\ref{theorem:omega} now gives
\[
\nu_e(\m)=m(q-1)(t-1),
\]
from which the result follows.
\end{proof}

\begin{prop}
\label{prop:fptQGor}
Let $R$ be an $F$-pure $\Q$-Gorenstein standard graded normal domain, over an $F$-finite field, with homogeneous maximal ideal $\m$. If $c$ is the order of $\omega_R$ in the divisor class group and $\omega^{(c)}$ is generated in degree $D$, then:
\[
\mathrm{fpt}(\m)=D/c.
\]
In particular, $\mathrm{fpt}(\m)=-a(R)$ if and only if $R$ is quasi-Gorenstein.
\end{prop}

\begin{proof}
For $q=p^e$ let us write $1-q=a(q)c+b(q)$, with $0\le b(q)<c$. By the assumptions we have:
\[
\omega_R^{(1-q)}=\left(\omega^{(c)}\right)^{a(q)}\omega_R^{(b(q))}.
\]
In particular, $\omega_R^{(1-q)}$ is generated in degrees $d$ satisfying:
\[
a(q)D+A\le d\le a(q)D+B,
\]
where the minimal generators of $\omega_R^{(b(q))}$ have degrees between $A$ and $B$. Therefore
\[
-a(q)D-B\le \nu_e(\m)\le -a(q)D-A,
\]
and $\mathrm{fpt}(\m)=\lim_{q\to\infty}\nu_e(\m)/q=D/c$.

For the last part of the statement, simply notice that, if $\omega_R^{(c)}$ is principal but $\omega_R$ is not, then the generator of $\omega_R^{(c)}$ must have degree less than $-a(R)c$, since $\omega_R^{(c)}\subseteq \omega_R^c$.
\end{proof}

\begin{rem}
In the above notation, if $c=p$ notice that $a(q)=-q/p$ and $b(q)=1$. Furthermore $A$ can be chosen to be the negative of the $a$-invariant of $R$, so:
\[
\nu_e(\m)\le (q/p)D+a(R)=(q-1)D/p+D/p+a(R)=(q-1)\mathrm{fpt}(\m)+\mathrm{fpt}(\m)+a(R).
\]
\end{rem}

Given a finitely generated graded $R$-module $M$, we denote by $\delta(M)$ the least integer $d$ such that $M_d\neq 0$. In the case in which $R$ is a normal domain, the canonical module $\omega_R$ is isomorphic (as a graded module) to a divisorial ideal of $R$. 

\begin{lem}
\label{lem1}
If $R$ is a normal standard graded domain, and $\ba$ a graded divisorial ideal, then:
\begin{enumerate}[\quad\rm(1)]
\item $\delta(\ba^{(-1)})\ge -\delta(\ba)$.
\item $\delta(\ba^{(-1)}) = -\delta(\ba)$ if and only if $\ba$ is principal. 
\end{enumerate}
\end{lem}

\begin{proof}
Let $\ba=(a_1,\ldots ,a_r)$, where the $a_i$ are homogeneous elements of the quotient field of $R$ of degree $d_i\in\Z$, where $\delta(\ba)=d_1\le d_2\le \dots\le d_r$. If $b\in \ba^{(-1)}$ is a homogeneous nonzero element of degree $l$, then $l+d_1\ge 0$ since $a_1b$ is a homogeneous nonzero element of $R$. This shows (1).

Concerning point (2), if $b$ is a homogeneous nonzero element of $\ba^{(-1)}$ of degree $-d_1$, then $ba_1=u$ is a unit of $R$. Because $ba_i=f_i\in R$ for all $i=1,\ldots ,r$, we have $a_i=u^{-1}f_ia_1$ for all $i=1,\ldots ,r$, so that $a_1$ generates $\ba$ as an $R$-module. 
\end{proof}

\begin{prop}
\label{prop:bound}
Let $R$ be an $F$-pure standard graded normal domain, over an $F$-finite field, with homogeneous maximal ideal $\m$. Then:
\[
\nu_e(\m)\le a(R)(1-q) \ \ \ \forall \ q=p^e.
\]
In particular $\mathrm{fpt}(\m)\le -a(R)$. Further, if $R$ is not quasi-Gorenstein, then:
\[
\nu_e(\m)<a(R)(1-q) \ \ \ \forall \ q=p^e.
\]
\end{prop}

\begin{proof}
For the first part of the statement, notice that $\delta\left(\omega_R^{(q-1)}\right)\le \delta\left(\omega_R^{q-1}\right)=\delta(\omega_R)(q-1)=-a(R)(q-1)$, so $\delta\left(\omega_R^{(1-q)}\right)\ge a(R)(q-1)$ by Lemma~\ref{lem1}, and $\nu_e(\m)\le a(R)(1-q)$ by Theorem~\ref{theorem:omega}. 

For the second part, assume that $\nu_e(\m)=a(R)(1-q)$ for some $q=p^e$. Since $\delta\left(\omega_R^{(q-1)}\right)\le a(R)(1-q)$, by putting together Theorem~\ref{theorem:omega} and Lemma~\ref{lem1},
$\omega_R^{(q-1)}$ must be principal and generated in degree $-a(R)(q-1)$. Notice that $\omega_R^{(q-1)}\supseteq \omega_R^{q-1}$ and $\delta\left(\omega_R^{q-1}\right)=a(R)(1-q)$. Therefore the only possibility is that $\omega_R^{(q-1)} = \omega_R^{q-1}$, so that $\omega_R$ must be principal itself.
\end{proof}

\begin{prop}
Let $R$ be an $F$-pure standard graded normal domain, over an $F$-finite field, with homogeneous maximal ideal $\m$. If the anti-canonical cover $\bigoplus_{k\ge 0}\omega_R^{(-k)}$ of $R$ is Noetherian, then $\mathrm{fpt}(\m)=-a(R)$ if and only if $R$ is quasi-Gorenstein. 
\end{prop}

\begin{proof}
If the anti-canonical cover of $R$ is Noetherian, then there exists a positive integer $c$ such that, if we write $1-q=-a(q)c-b(q)$ with $a(q)$ positive and $0\le b(q)<c$:
\[
\omega_R^{(1-q)}=\left(\omega^{(-c)}\right)^{a(q)}\omega_R^{(-b(q))}.
\]
Let us say that $\omega^{(-c)}$ is generated in degrees $-d_1< \dots < -d_r$. Furthermore, let $-e_1< \dots < -e_s$ be the degrees of the minimal generators of $R,\omega^{(-1)},\ldots ,\omega^{(-c+1)}$. We have that
\[
\nu_e(\m)\in\left\{ \big(\sum_{i=1}^r n_i d_i\big) +e_j \ \bigg| \ n_i \in \Z_{\ge 0} \ \mbox{with} \ \sum_{i=1}^r n_i=a(q)
\ \mbox{and} \ j=1,\ldots ,s\right\}, 
\]
and in particular, $\nu_e(\m) \le a(q)d_1+e_1$ for every $q=p^e$. 
Therefore, 
\[\mathrm{fpt}(\m)=\lim_{q \to \infty} \frac{\nu_e(\m)}{q} \le \frac{d_1}{c}.\]
If $R$ is not quasi-Gorenstein, then it follows from the same argument of Proposition~\ref{prop:bound} that $-d_1>a(R)c$, which implies that $\mathrm{fpt}(\m)<-a(R)$. 
\end{proof}

\begin{rem}
When the assumptions of the proposition are satisfied and $R$ is strongly F-regular, $\mathrm{fpt}(\m)$ is known to be a rational number by \cite[Theorem~B]{CEMS}.
\end{rem}

\begin{prop}
Let $R$ be an $F$-pure standard graded normal domain, over an $F$-finite field, with homogeneous maximal ideal $\m$. If there exists a positive integer $c$ such that $\delta\left(\omega^{(c)}\right)<-a(R)c$, then $\mathrm{fpt}(\m)<-a(R)$. 
\end{prop}

\begin{proof}
Let $\delta\left(\omega^{(c)}\right)=D$. With the same notation of the proof above $q-1=a(q)c+b(q)$, so:
\[
\delta(\omega_R^{(q-1)})\le a(q)D+\delta\left(\omega_R^{(b(q))}\right).
\]
Then, using Theorem~\ref{theorem:omega} together with Lemma~\ref{lem1}, for such $q$:
\[
\nu_e(\m)\le a(q)D.
\]
Thus, $\mathrm{fpt}(\m)=\lim_{q\to\infty}\nu_e(\m)/q\le D/c<-a(R)$.
\end{proof}

The following provides strong evidence for the conjecture of Hirose-Watanabe-Yoshida~\ref{conj:HWY} and, more generally, for the standard graded case of Conjecture~\ref{conj:fpt=fit}.

\begin{cor}
\label{cor:final}
Let $R$ be an $F$-pure standard graded normal domain, over an $F$-finite field, with homogeneous maximal ideal $\m$. Suppose that one of the following is satisfied: 
\begin{enumerate}[\quad\rm(1)]
\item The anti-canonical cover $\bigoplus_{k\ge 0}\omega_R^{(-k)}$ of $R$ is noetherian.
\item For some positive integer $c$, there is a nonzero element of $\omega_R^{(c)}$ of degree $<-a(R)c$. 
\end{enumerate}
Then $\mathrm{fpt}(\m)=-a(R)$ if and only if $R$ is quasi-Gorenstein.
\end{cor}

The following gives an extension of \cite[Proposition~3.4]{HWY}.

\begin{prop}
\label{prop:critical value}
Let $R$ be an $F$-pure standard graded normal domain, over an $F$-finite field of characteristic $p>0$, with homogeneous maximal ideal $\m$. Suppose there exists a positive integer $c$ not divisible by $p$ such that the minimal generators of $\omega_R^{(-c)}$ have equal degree, and $\omega_R^{(-ck)}=\big(\omega_R^{(-c)}\big)^k$ for each $k\in\N$. Then $\nu_e(\m)=(p^e-1)\mathrm{fpt}(\m)$ for infinitely many positive integers $e$. In particular, $(R, \m^{\mathrm{fpt}(\m)})$ is sharply $F$-pure. 
\end{prop}

\begin{proof}
Since $p$ does not divide $c$, there exists an infinite subset $A\subseteq \{p^e \mid e\in\N\}$ such that $q-1=a(q)c$ for all $q\in A$, with $a(q)\in\N$. For such $q$, 
\[
\omega_R^{(1-q)}=\left(\omega_R^{(-c)}\right)^{a(q)}.
\]
Therefore, if $\omega^{(-c)}$ is generated in degree $-d$, then $\nu_e(\m)=a(q)d$ for all $q \in A$, 
so that $\mathrm{fpt}(\m)=d/c$ and $\nu_e(\m)=(q-1)\mathrm{fpt}(\m)$ for all $q \in A$. 
\end{proof}

\section{Characteristic zero case}

Throughout this section, let $X$ be a normal variety over an algebraically closed field of characteristic zero and $\ba$ be a nonzero coherent ideal sheaf on $X$. 

We prove a characteristic zero analogue of Conjecture~\ref{conj:HWY}. 
First, we define a variant of multiplier submodules:

\begin{defn}
Let $\pi \colon Y \to X$ be a \textit{log resolution} of $(X, \ba)$, that is, $\pi$ is a proper birational morphism from a smooth variety $Y$ such that $\ba \sO_Y=\sO_Y(-F)$ is invertible and $\mathrm{Exc}(\pi) \cup \mathrm{Supp}(F)$ is a simple normal crossing divisor. 
Let $E$ be the reduced divisor supported on $\mathrm{Exc}(\pi)$. 
For a real number $t \ge 0$, the \textit{multiplier submodule} $\mathcal{J}(\omega_X, \ba^t)$ is defined by 
\[
\mathcal{J}(\omega_X, \ba^t):=\pi_*\omega_Y(\lceil -tF \rceil) \subseteq \omega_X.
\]
This submodule of $\omega_X$ is independent of the choice of $\pi$, see, for example, the proof of \cite[Proposition~3.4]{ST}. 
When $\ba=\sO_X$ or $t=0$, we simply denote $\mathcal{J}(\omega_X, \ba^t)$ by $\mathcal{J}(\omega_X)$. 

As a variant of $\mathcal{J}(\omega_X, \ba^t)$, we define the submodule $\mathcal{I}(\omega_X, \ba^t)$ of $\omega_X$ by
\[
\mathcal{I}(\omega_X, \ba^t):=\left\{\begin{array}{ll}\pi_*\omega_Y(\lceil \epsilon E-(t-\epsilon)F \rceil) & \textup{if $t>0$}\\
\pi_*\omega_Y(E) & \textup{if $t=0$}
\end{array}
\right.
\]
for sufficiently small $\epsilon>0$. 
It is easy to see that $\mathcal{I}(\omega_X, \ba^t)$ is independent of the choice of $\epsilon$ if $\epsilon>0$ is sufficiently small. 
When $\ba=\sO_X$ or $t=0$, we simply denote $\mathcal{I}(\omega_X, \ba^t)$ by $\mathcal{I}(\omega_X)$. 
\end{defn}

\begin{lem}
$\mathcal{I}(\omega_X, \ba^t)$ is independent of the choice of the resolution. 
\end{lem}

\begin{proof}
Although it immediately follows from \cite[Lemma~13.3 and Corollary~13.7]{FST}, we give a more direct proof here. 

We consider the case where $t>0$; the case $t=0$ follows from a similar argument. 
Let $f \colon Y \to X$ be a log resolution of $(X, \ba)$ such that $\ba \sO_X=\sO_X(-F)$ is invertible, and let $E_Y$ be the reduced divisor supported on $\mathrm{Exc}(f)$. 
Let $g \colon Z \to Y$ be a log resolution of~$(Y, E_Y+F)$, and let $E_Z$ be the reduced divisor supported on $\mathrm{Exc}(g)$. 
Then it is enough to show that 
\[
\omega_Y(\lceil \epsilon E_Y-(t-\epsilon) F\rceil)=g_*\omega_Z(\lceil \epsilon'(g^{-1}_*E_Y+E_Z)-(t-\epsilon')g^*F \rceil)
\]
for sufficiently small real numbers $\epsilon, \epsilon'>0$, since two log resolutions of $(X, \ba)$ can be dominated by a third log resolution.
Let $\bigcup_i E_i$ be the irreducible decomposition of $\Supp (E_Y+F)$. 
For sufficiently small $\epsilon>0$, we can write 
\[
\lceil K_Y+ \epsilon E_Y-(t-\epsilon) F\rceil=K_Y- tF+\sum_{i}a_i E_i,
\]
where $1 \ge a_i>0$ for all $i$. 
Since $\sum_i E_i$ is a simple normal crossing divisor on $Y$, the pair $(Y, \sum_i a_i E_i)$ is log canonical. By the definition of log canonical pairs, we have
\begin{align*}
G:=& \lceil K_Z+\epsilon' (g^{-1}_*E_Y+E_Z)-(t-\epsilon')g^*F \rceil- g^*\lceil K_Y+ \epsilon E_Y-(t-\epsilon) F\rceil\\
= & \left\lceil K_{Z/Y}+\epsilon' \left(g^{-1}_*E_Y+E_Z+g^*F\right)-g^*\sum_i a_i E_i \right\rceil \ge 0.
\end{align*}
Note that $G$ is a $g$-exceptional divisor for sufficiently small $\epsilon'>0$. 
Therefore, 
\begin{align*}
g_*\omega_Z(\lceil \epsilon'(g^{-1}_*E_Y+E_Z)-(t-\epsilon')g^*F \rceil)
=& g_*\left(g^*\left(\omega_Y(\lceil \epsilon E_Y-(t-\epsilon) F\rceil\right)) \otimes \sO_Z(G) \right)\\
=& \omega_Y(\lceil \epsilon E_Y-(t-\epsilon) F\rceil) \otimes g_*\sO_Z(G)\\
=& \omega_Y(\lceil \epsilon E_Y-(t-\epsilon) F\rceil).\qedhere
\end{align*}
\end{proof}

\begin{rem}
\label{rem:rational, db}
\begin{enumerate}[\quad\rm(1)]
\item $X$ has only rational singularities if and only if $X$ is Cohen-Macaulay and $\mathcal{J}(\omega_X)=\omega_X$, see \cite[Theorem~5.10]{KM}.

\item If $X$ has only Du Bois singularities, then $\mathcal{I}(\omega_X)=\omega_X$, see \cite{KSS}.
In case $X$ is Cohen-Macaulay, the converse holds as well. 
The reader is referred to \cite{Sch1} for the definition and a simple characterization of Du Bois singularities. 
\end{enumerate}
\end{rem}

Using $\mathcal{I}(\omega_X, \ba^t)$, we define a new invariant of singularities in characteristic zero:

\begin{defn}
\label{def:db threshold}
Suppose that $X$ has only Du Bois singularities. 
Then the threshold $\mathrm{d}(\ba)$ is defined by 
\[
\mathrm{d}(\ba):=\sup\{t \ge 0 \mid \mathcal{I}(\omega_X, \ba^t)=\omega_{X}\}.
\]
If $x$ is a closed point of $X$, then the threshold $\mathrm{d}_x(\ba)$ is defined by 
\[
\mathrm{d}_x(\ba):=\sup\{t \ge 0 \mid \mathcal{I}(\omega_X, \ba^t)_x=\omega_{X,x}\}.
\]
\end{defn}

\begin{rem}
\label{rem:graded char0}
Let $R$ be an $\N$-graded ring with $R_0$ an algebraically closed field $k$ of characteristic zero, and let~$\m$ be the homogeneous maximal ideal of $R$. 
Let $X=\Spec R$ and $x \in X$ be the closed point corresponding to $\m$. 
By considering a $k^*$-equivariant log resolution of $(X, \m)$, we see that $\mathcal{I}(\omega_X, \m^t)$ is a graded submodule of the graded canonical module $\omega_R$. 
This implies that $\mathrm{d}(\m)=\mathrm{d}_x(\m)$. 
\end{rem}

In \cite{dFH} de~Fernex-Hacon extended the notion of log canonical thresholds to the non-$\Q$-Gorenstein setting; we recall their definition:

\begin{defn}[{\cite[Proposition~7.2]{dFH}}]
Suppose that $t \ge 0$ is a real number. 
\begin{enumerate}[\quad\rm(1)]
\item The pair $(X, \ba^t)$ is said to be \textit{klt} (resp. \textit{log canonical}) in the sense of de~Fernex-Hacon if there exists an effective $\Q$-divisor $\Delta$ on $X$ such that $K_X+\Delta$ is $\Q$-Cartier and $((X, \Delta); \ba^t)$ is klt (resp. log canonical) in the classical sense. 
That is, if $\pi \colon Y \to X$ is a log resolution of $(X, \Delta, \ba)$ such that $\ba\sO_Y=\sO_Y(-F)$ is invertible and if we write 
\[
K_Y-\pi^*(K_X+\Delta)-tF=\sum_{i}a_i E_i,
\]
where the $E_i$ are prime divisors on $Y$ and the $a_i$ are real numbers, then $a_i >-1$ (resp. $a_i \ge -1$) for all $i$. 
The \textit{log canonical threshold} $\mathrm{lct}(\ba)$ of $\ba$ is defined by 
\[
\hspace*{3em} \mathrm{lct}(\ba):=\sup\{t \ge 0 \mid (X, \ba^t) \textup{ is log canonical in the sense of de~Fernex-Hacon} \}.
\]

\item Let $x$ be a closed point of $X$. Then $(X, \ba^t)$ is \textit{klt} (resp. \textit{log canonical}) at $x$ in the sense of de~Fernex-Hacon if there exists a open neighborhood $U$ of $x$ such that $(U, (\ba|_U)^t)$ is klt (resp. log canonical) in the sense of de~Fernex-Hacon. 
If, in addition, $\ba=\sO_X$, then we say that $(X, x)$ is a \textit{log terminal} (resp. \textit{log canonical}) singularity in the sense of de~Fernex-Hacon. 
The \textit{log canonical threshold} $\mathrm{lct}_x(\ba)$ of $\ba$ at~$x$ is defined by 
\[
\hspace*{3.5em} \mathrm{lct}_x(\ba):=\sup\{t \ge 0 \mid (X, \ba^t) \textup{ is log canonical at $x$ in the sense of de~Fernex-Hacon} \}.
\]
If $\Delta$ is an effective $\Q$-divisor on $X$ such that $K_X+\Delta$ is $\Q$-Cartier and $(X, \Delta)$ is log canonical at $x$ (in the classical sense), then the \textit{log canonical threshold} $\mathrm{lct}_x(\Delta;\ba)$ is defined by 
\[
\hspace*{3.5em} \mathrm{lct}_x(\Delta; \ba):=\sup\{t \ge 0 \mid ((X, \Delta);\ba^t) \textup{ is log canonical at $x$ (in the classical sense)} \}.
\]
\end{enumerate}
\end{defn}

We prove some basic properties of $\mathrm{d}_x(\ba)$.

\begin{lem}
\label{lem:dbt basic}
Let $(X, x)$ be a $d$-dimensional normal singularity. 
\begin{enumerate}[\quad\rm(1)]
\item If $(X, x)$ is a rational singularity, then 
\[
\mathrm{d}_x(\ba)=\sup\{t \ge 0 \mid \mathcal{J}(\omega_X, \ba^t)_x=\omega_{X,x}\}.
\]

\item Suppose that $(X, x)$ is a log canonical singularity in the sense of de~Fernex-Hacon. Then $\mathrm{lct}_x(\ba) \le \mathrm{d}_x(\ba)$. In addition, if $X$ is quasi-Gorenstein at $x$, then $\mathrm{lct}_x(\ba) = \mathrm{d}_x(\ba)$.

\item Suppose that $(X,x)$ is a Cohen-Macaulay Du Bois singularity. 
Then $\mathrm{d}_x(\m_x) \le d$. 
If $J \subseteq \sO_{X, x}$ is a minimal reduction of the maximal ideal $\m_x$, then $\m_x^{d+1-\lceil \mathrm{d}_x(\m_x) \rceil} \subseteq J$. 

\item Suppose that $X$ is Cohen-Macaulay at $x$. 
If $\mathrm{d}_x(\m_x)>d-1$, then $X$ is nonsingular at $x$ and in particular $\mathrm{d}_x(\m_x)=d$. 
\end{enumerate}
\end{lem}

\begin{proof}
(1) Shrinking $X$ if necessary, we may assume that $X$ has only rational singularities. 
First we will check that $\mathrm{d}_x(\ba)>0$. 
Let $\pi \colon Y \to X$ be a log resolution of $(X, \ba)$ such that $\ba \sO_Y=\sO_Y(-F)$ is invertible and let $E$ be the reduced divisor supported on $\mathrm{Exc}(\pi)$. 
For sufficiently small $t>\epsilon>0$, 
\[
K_Y+E \ge \lceil K_Y+\epsilon E-(t-\epsilon)F \rceil \ge K_Y.
\]
Taking the pushforward $\pi_*$, we obtain inclusions $\omega_X \supset \mathcal{I}(\omega_X, \ba^t) \supset \mathcal{J}(\omega_X)=\omega_X$ by Remark~\ref{rem:rational, db}~(1). 
That is, $\mathrm{d}_x(\ba) \ge t>0$. 

Now we will show the assertion. Since $\mathcal{J}(\omega_X, \ba^t) \subseteq \mathcal{I}(\omega_X, \ba^t)$, the inequality 
\[
\mathrm{d}_x(\ba) \ge \sup\{t \ge 0 \mid \mathcal{J}(\omega_X, \ba^t)_x=\omega_{X,x}\}
\]
is obvious. 
We will prove the reverse inequality. 
It is enough to show that if $\mathcal{I}(\omega_X, \ba^t)_x=\omega_{X,x}$ with $t>0$, then $\mathcal{J}(\omega_X, \ba^{t-\epsilon})_x=\omega_{X,x}$ for all $t \ge \epsilon>0$. 
Fix a real number $t \ge \epsilon>0$. 
Shrinking $X$ again if necessary, we may assume that $X$ is affine and that $\mathcal{I}(\omega_X, \ba^t)=\omega_{X}$. 
This means that for sufficiently small $(1/2)\epsilon \ge \epsilon'>0$, 
\[
\mathrm{ord}_{E_i}(K_Y+\mathrm{div}_Y(f)+\epsilon' E-(t-\epsilon')F)>-1
\]
for every prime divisor $E_i$ on $Y$ and for every $f \in \omega_X$. 
If $E_i$ is an irreducible component of $\Supp F$, then 
\[
\mathrm{ord}_{E_i}(K_Y+\mathrm{div}_Y(f)-(t-\epsilon)F) \ge \mathrm{ord}_{E_i}(K_Y+\mathrm{div}_Y(f)+\epsilon'E-(t-\epsilon')F)>-1.
\]
On the other hand, since $X$ has only rational singularities, $K_Y+\mathrm{div}_Y(f) \ge 0$ by Remark~\ref{rem:rational, db}~(1). 
Therefore, if $E_i$ is not a component of $\Supp F$, then 
\[
\mathrm{ord}_{E_i}(K_Y+\mathrm{div}_Y(f)-(t-\epsilon)F)= \mathrm{ord}_{E_i}(K_Y+\mathrm{div}_Y(f)) \ge 0.
\]
Summing up above, we conclude that $\lceil K_Y+\mathrm{div}_Y(f)-(t-\epsilon)F \rceil \ge 0$ for every $f \in \omega_X$, that is, $\mathcal{J}(\omega_X, \ba^{t-\epsilon})=\omega_X$. 

(2) For the former assertion, it is enough to show that if $(X, \ba^t)$ is log canonical at $x$ in the sense of de~Fernex-Hacon, then $\mathcal{I}(\omega_X, \ba^t)_x=\omega_{X, x}$. 
Shrinking $X$ if necessary, we may assume that $X$ is affine and there exists an effective $\Q$-divisor $\Delta$ on $X$ such that $((X, \Delta);\ba^t)$ is log canonical with $K_X+\Delta$ $\Q$-Cartier of index $r$. 
Let $\pi \colon Y \to X$ be a log resolution of $(X, \Delta, \ba)$ such that $\ba \sO_Y=\sO_Y(-F)$ is invertible and let $E$ be the reduced divisor supported on $\mathrm{Exc}(\pi)$. 
By the definition of log canonical pairs, 
\[
\lceil K_Y-\pi^*(K_X+\Delta)+\epsilon_1 E-(t-\epsilon_2)F \rceil \ge 0
\]
for every $\epsilon_1>0$ and $t \ge \epsilon_2>0$ (when $t=0$, put $\epsilon_2=0$). Let $f \in \omega_X$. 
Since the fractional ideal $\omega_X^r$ is contained in $\sO_X(r(K_X+\Delta))$, one has $\mathrm{div}_Y(f)+\pi^*(K_X+\Delta) \ge 0$. 
It follows from these two inequalities that 
\[
\lceil K_Y+\epsilon_1 E-(t-\epsilon_2)F\rceil+\mathrm{div}_Y(f) \ge 0,
\]
which implies that $f \in \mathcal{I}(\omega_X, \ba^t)$. 

Now we will show the latter assertion. Shrinking $X$ again if necessary, we may assume that $\omega_X \cong \sO_X$. 
Then $\mathcal{I}(\omega_X, \ba^t)$ can be identified with the maximal non-lc ideal $\mathcal{J}'(X, \ba^t)$ under this isomorphism (see \cite[Definition~7.4]{FST} for the definition of $\mathcal{J}'(X, \ba^t)$). 
Since $\mathcal{J}'(X, \ba^t)=\sO_X$ if and only if $(X, \ba^t)$ is log canonical by the definition of $\mathcal{J}'(X, \ba^t)$, one has the equality that $\mathrm{lct}_x(\ba) = \mathrm{d}_x(\ba)$. 

(3) The proof is essentially the same as that of \cite[Theorem~5.2.5]{Sh}. 
Let $f \colon Y \to \Spec \sO_{X, x}$ be the blow-up at $\m_x$ with exceptional divisor $F_x$. 
Take a log resolution $\pi \colon \widetilde{X} \to \Spec \sO_{X, x}$ of $\m_x$. Then there exists a morphism $g \colon \widetilde{X} \to Y$ such that $\pi=f \circ g$. 
Let $E=\sum_{i=1}^s E_i$ be the reduced divisor supported on $\mathrm{Exc}(\pi)$. 
We may assume that $E_1, \dots, E_r$ are all the components of $E$ dominating an irreducible component of $F_x$, and put $E':=\sum_{i=r+1}^s E_i$. 
If $t>d$, then 
\begin{align*}
\mathcal{I}(\omega_{X}, \m_x^{t})_x &= \pi_*\sO_{\widetilde{X}}(\lceil K_{\widetilde{X}}+\epsilon E-(t-\epsilon) g^*F_x \rceil)\\
& \subseteq \pi_*\sO_{\widetilde{X}}( K_{\widetilde{X}}+E'-d g^*F_x )\\
& = f_*\left(g_*\sO_{\widetilde{X}}(K_{\widetilde{X}}+E')\otimes \sO_Y(-dF_x)\right)\\
& \subseteq f_*\omega_Y(-dF_x)\\
&=f_*\m_x^d \omega_Y
\end{align*}
for sufficiently small $\epsilon>0$. 
It follows from \cite[Theorem~3.7]{HuTr} and \cite[Lemma~5.1.6]{HS} that 
\[
\mathcal{I}(\omega_{X}, \m_x^{t})_x :_{\sO_{X, x}} \omega_{X,x} \subseteq f_*\m_x^d \omega_Y:_{\sO_{X, x}}\omega_{X, x}=\mathrm{core}(\m), 
\]
where $\mathrm{core}(\m)$ is the intersection of all reductions of $\m$. 
In particular, $\mathcal{I}(\omega_{X}, \m_x^{t})_x \subsetneq \omega_{X, x}$, that is, $\mathrm{d}_x(\m_x)<t$. 
Thus, $\mathrm{d}_x(\m_x) \le d$. 

Since $\mathcal{I}(\omega_{X}, \m_x^{\mathrm{d}_x(\m_x)-\epsilon})_x=\omega_{X, x}$ for every $\epsilon>0$ (we put $\epsilon=0$ when $\mathrm{d}_x(\m_x)=0$), 
by the same argument as above, 
\[
\m_x^{d+1-\lceil \mathrm{d}_x(\m_x) \rceil} \omega_{X, x} = \m_x^{d+1-\lceil \mathrm{d}_x(\m_x) \rceil} \mathcal{I}(\omega_{X}, \m_x^{\mathrm{d}_x(\m_x)-\epsilon})_x 
\subseteq f_*\m_x^d \omega_Y
\]
for sufficiently small $\epsilon>0$. 
It then follows from \cite[Theorem~3.7]{HuTr} and \cite[Lemma~5.1.6]{HS} again that 
\[
\m_x^{d+1-\lceil \mathrm{d}_x(\m_x) \rceil} \subseteq f_*\m_x^d \omega_Y:_{\sO_{X, x}}\omega_{X, x}=\mathrm{core}(\m) \subseteq J. 
\]

(4) Let $J$ be a minimal reduction of $\m_x$. 
It then follows from (3) that $\m_x=J$, which means that $\m_x$ is generated by at most $d$ elements, 
that is, $X$ is nonsingular at $x$. 
If $X$ is nonsingular at $x$, then by (2), we see that $\mathrm{d}_x(\m_x)=\mathrm{lct}_x(\m_x)=d$. 
\end{proof}

We can compute the log canonical threshold of the maximal ideal of an affine determinantal variety 
using $F$-pure thresholds:

\begin{prop}
Let $D:=\Spec k[T]/I$ be the affine determinantal variety over an algebraically closed field $k$ of characteristic zero, where $T$ is an $m \times n$ matrix of indeterminates with $m \le n$, and $I$ is the ideal generated by the size $t$ minors of~$T$ where $1\le t\le m$.
Let $\m$ be the homogeneous maximal ideal of $k[T]/I$, corresponding to the origin $0$ in $D$. 
Then
\[
\mathrm{lct}(\m) = m(t-1).
\]
\end{prop}

\begin{proof}
For each prime integer $p$, let $R_p:=\F_p[T]/I_{p}$ be the modulo $p$ reduction of $k[T]/I$, and $\m_p$ the homogeneous maximal ideal of $R_p$. 
It then follows from \cite[Theorem~6.4]{CEMS} and Proposition~\ref{prop:determinantal} that 
\[
\mathrm{lct}(\m)=\lim_{p \to \infty} \mathrm{fpt}(\m_p)=m(t-1).\qedhere
\]
\end{proof}

\begin{prop}
\label{prop:non-Gorenstein2}
Let $x$ be a closed point of $X$ and $\Delta$ be an effective $\Q$-divisor on $X$ such that $(X, \Delta)$ is log canonical at $x$ with $K_X+\Delta$ being $\Q$-Cartier of index $r$. 
If $X$ is not quasi-Gorenstein at $x$, then 
\[
\mathrm{lct}_x(\Delta; \m_x)+\frac{1}{r} \le \mathrm{d}_x(\m_x).
\]
In particular, if $\mathrm{d}_x(\m_x)=0$, then $X$ is quasi-Gorenstein at $x$. 
\end{prop}

\begin{proof}
Shrinking $X$ if necessary, we may assume that $X$ is affine, $\sO_X(r(K_X+\Delta)) \cong \sO_X$ and $((X, \Delta); \m_x^{\mathrm{lct}_x(\Delta;\m_x)})$ is log canonical. 
Let $\pi \colon Y \to X$ be a log resolution of $(X, \Delta, \m_x)$ such that $\m_x \sO_X=\sO_X(-F_x)$ is invertible, and let $E$ be the reduced divisor supported on $\mathrm{Exc}(\pi)$. 
Putting $t=\mathrm{lct}_x(\Delta; \m_x)$, one has the inequality 
\[
\lceil K_Y-\pi^*(K_X+\Delta)-(t-\epsilon)F_x+\epsilon E \rceil \ge 0 
\]
for every $\epsilon>0$. 
On the other hand, since $X$ is not quasi-Gorenstein at $x$, the fractional ideal $\omega_{X}^r$ is contained in $\m_x \sO_X(r(K_X+\Delta))$. 
Hence, for each $f \in \omega_X$, one has the inequality 
\[
r\mathrm{div}_Y(f) +r\pi^*(K_X+\Delta)-F_x \ge 0.
\]
It follows from these two inequalities that 
\begin{align*}
0 \le & \lceil K_Y-\pi^*(K_X+\Delta)-(t-\epsilon)F_x+\epsilon E \rceil\\
= & \left\lceil K_Y+ \epsilon E-\left(t+\frac{1}{r}-\epsilon \right)F_x - \pi^*(K_X+\Delta)+\frac{1}{r}F_x \right\rceil\\
\le & K_Y+\left\lceil \epsilon E-\left(t+\frac{1}{r}-\epsilon \right)F_x \right\rceil+\mathrm{div}_Y(f)
\end{align*}
for all $\epsilon>0$ and all $f \in \omega_X$. 
This means that $\mathcal{I}(\omega_X, \m_x^{\mathrm{lct}_x(\Delta; \m_x)+1/r})=\omega_X$, that is, $\mathrm{d}_x(\m_x) \ge \mathrm{lct}_x(\Delta; \m_x)+1/r$. 
\end{proof}

The following theorem is the main result of this section; this is a characteristic zero analogue of Theorem~\ref{thm:anti-canonical}. 

\begin{thm}
\label{thm:lc}
Suppose that $(X, x)$ is a log canonical singularity in the sense of de~Fernex-Hacon. 
Assume in addition that the anti-canonical cover $\bigoplus_{n \ge 0}\sO_X(-nK_X)_x$ is Noetherian. 
Then $\mathrm{lct}_x(\m_x)=\mathrm{d}_x(\m_x)$ if and only if $(X, x)$ is quasi-Gorenstein. 
\end{thm}

\begin{proof}
Since the ``if" part immediately follows from Lemma~\ref{lem:dbt basic}~(2), we will show the ``only if" part. 
Shrinking $X$ if necessary, we may assume that $X$ is log canonical in the sense of de~Fernex-Hacon and that $\bigoplus_{n \ge 0}\sO_X(-nK_X)$ is Noetherian. 
Then one can find an integer $r \ge 1$ such that $\sO_X(-rmK_X)=\sO_X(-rK_X)^m$ for every integer $m \ge 1$. 
Fix a real number $\epsilon$ with $\min\{\mathrm{lct}(\m_x), 1/r\}>\epsilon>0$; when $\mathrm{lct}(\m_x)=0$, put $\epsilon=0$. 
Since $(X, \m_x^{\mathrm{lct}(\m_x)-\epsilon})$ is log canonical in the sense of de~Fernex-Hacon, there exists an integer $m_0 \ge 1$ such that the $m$-th limiting log discrepancy $a_{m, F}(X, \m_x^{\mathrm{lct}(\m_x)-\epsilon})$ is nonnegative for every prime divisor $F$ over $X$ and for every positive multiple $m$ of $m_0$ by \cite[Definition~7.1]{dFH} (see \textit{loc.~cit.} for the definition of the $m$-th limiting log discrepancy of a pair). By the choice of $r$, one has 
\[
a_{r, F}(X, \m_x^{\mathrm{lct}(\m_x)-\epsilon})=a_{rm_0, F}(X, \m_x^{\mathrm{lct}(\m_x)-\epsilon}) \ge 0.
\]
It follows from an argument similar to the proof of \cite[Proposition~7.2]{dFH} that there exists an effective $\Q$-divisor $\Delta$ on $X$ such that $K_X+\Delta$ is $\Q$-Cartier of index $r$ and $((X, \Delta); \m_x^{\mathrm{lct}(\m_x)-\epsilon})$ is log canonical. 
If $X$ is not quasi-Gorenstein at $x$, then by Proposition~\ref{prop:non-Gorenstein2}, 
\[
\mathrm{lct}_x(\m_x) <\mathrm{lct}_x(\m_x)-\epsilon +\frac{1}{r} \le \mathrm{lct}_x(\Delta; \m_x) +\frac{1}{r} \le \mathrm{d}_x(\m_x).
\]
This contradicts the assumption that $\mathrm{lct}_x(\m_x)=\mathrm{d}_x(\m_x)$. 
\end{proof}

\begin{cor}
\label{cor:log terminal}
Suppose that $(X, x)$ is a log terminal singularity in the sense of de~Fernex-Hacon. 
Then 
$\mathrm{lct}_x(\m_x)=\mathrm{d}_x(\m_x)$ if and only if $(X, x)$ is Gorenstein. 
\end{cor}

\begin{proof}
Since $(X, x)$ is log terminal, using the minimal model program for klt pairs, one can show that the anti-canonical cover $\bigoplus_{n \ge 0}\sO_X(-nK_X)_x$ is Noetherian (see \cite[Theorem~92]{Ko}). Thus, the assertion follows from Theorem~\ref{thm:lc}. 
\end{proof}

\begin{prop}
\label{prop:dbt=-a}
Let $R$ be a normal standard graded ring, with $R_0$ an algebraically closed field of characteristic zero, and let $\m$ denote the homogeneous maximal ideal of $R$. 
Suppose that $\Spec R$ has only Du Bois singularities. 
Then $\mathrm{d}(\m) \le -a(R)$. 
\end{prop}

\begin{proof}
Put $X=\Spec R$. 
Since $X$ has only Du Bois singularities, $a(R) \le 0$ by \cite[Theorem~4.4]{Ma}. 
Suppose that $\mathcal{I}(\omega_X, \m^t)=\omega_X$ with $t>0$. 
Let $\phi \colon Y \to X$ be the blow-up of $X$ at $\m$ and $E=\Proj R$ be its exceptional divisor. 
Take a log resolution $\psi \colon \widetilde{X} \to Y$ of $(Y, E)$ and let $\widetilde{E}$ be the strict transform of $E$ on $\widetilde{X}$. 
We fix a canonical divisor $K_{\widetilde{X}}$ on $\widetilde{X}$ such that $\psi_*K_{\widetilde{X}}=K_Y$. 
Since $\mathcal{I}(\omega_X, \m^t)=\omega_X$, 
\[
\mathrm{ord}_{\widetilde{E}}(\lceil K_{\widetilde{X}}+\mathrm{div}_{\widetilde{X}}(f)+\epsilon \widetilde{E}-(t-\epsilon) \psi^*E \rceil) \ge 0
\]
for all $f \in \omega_X$ and all sufficiently small $\epsilon>0$. 
Taking the direct image by $\psi$, we see that $\mathrm{ord}_{E}(\lceil K_{Y}+\mathrm{div}_{Y}(f)+\epsilon E-(t-\epsilon) E \rceil) \ge 0$, 
that is, $\phi_*\omega_{Y}(\lceil \epsilon-t \rceil E)=\omega_X$ for sufficiently small $\epsilon>0$. 
On the other hand, it is easy to see by the definition of $\phi$ (see, for example, \cite[Proposition~6.2.1]{HS}) that 
\[
\phi_*\omega_{Y}(\lceil \epsilon-t \rceil E)=[\omega_X]_{\ge \lfloor t- \epsilon \rfloor+1}. 
\]
Thus, $t \le -a(R)$, that is, $\mathrm{d}(\m) \le -a(R)$. 
\end{proof}

As a consequence, we can prove a characteristic zero analogue of Conjecture~\ref{conj:HWY}, which gives an affirmative answer to \cite[Conjecture~6.9]{dSNB}.

\begin{cor}
\label{cor:char 0 main thm}
Let $R$ be a normal standard graded ring, with $R_0$ an algebraically closed field of characteristic zero. Let $\m$ denote the homogeneous maximal ideal of $R$. 
Assume that $X:=\Spec R$ has log canonical singularities in the sense of de~Fernex-Hacon. 
\begin{enumerate}[\quad\rm(1)]
\item Then $\mathrm{lct}(\m) \le -a(R)$.

\item Suppose in addition that the anti-canonical cover $\bigoplus_{n \ge 0}\sO_X(-nK_X)$ is Noetherian (this assumption is satisfied, for example, if $X$ has log terminal singularities in the sense of de~Fernex-Hacon or if $R$ is $\Q$-Gorenstein). 
Then $\mathrm{lct}(\m)=-a(R)$ if and only if $R$ is quasi-Gorenstein. 
\end{enumerate}
\end{cor}

\begin{proof}
Since (1) follows from Remark~\ref{rem:graded char0}, Lemma~\ref{lem:dbt basic}~(2) and Proposition~\ref{prop:dbt=-a}, we will show (2). 
Let $x \in X$ be the closed point corresponding to $\m$. 
If $\mathrm{lct}(\m)=-a(R)$, then $\mathrm{lct}_x(\m)$ has to be equal to $\mathrm{d}_x(\m)$ by Remark~\ref{rem:graded char0}, Lemma~\ref{lem:dbt basic}~(2) and Proposition~\ref{prop:dbt=-a} again. 
It follows from Theorem~\ref{thm:lc} that $X$ is quasi-Gorenstein at $x$, which is equivalent to saying that $R$ is quasi-Gorenstein. 

Next we will show the ``if" part of (2). Suppose that $R$ is quasi-Gorenstein. 
Let $\phi \colon Y \to X$ be the blow-up of $X=\Spec R$ at $\m$ and $E=\Proj R$ be its exceptional divisor. 
Note that $Y$ is normal and quasi-Gorenstein. 
It is easy to see that $K_{Y/X}=-(1+a(R))E$, see, for example, the proof of \cite[Proposition~5.4]{SS}. 
Take a log resolution $\psi \colon \widetilde{X} \to Y$ of $(Y, E)$, and then 
\[
K_{\widetilde{X}/X}+a(R)\psi^*E=K_{\widetilde{X}/Y}+\psi^*(K_{Y/X}+a(R)E)=K_{\widetilde{X}/Y}-\psi^*E.
\]
Since $X$ has only log canonical singularities, $E$ has also only log canonical singularities. 
It follows from inversion of adjunction for log canonical pairs \cite{Ka} that $(Y, E)$ is log canonical, which implies that all the coefficients of the divisor $K_{\widetilde{X}/Y}-\psi^*E$ are greater than or equal to $-1$. 
Thus, $(X, \m^{-a(R)})$ is log canonical, that is, $\mathrm{lct}(\m) \ge -a(R)$. 
\end{proof}

\begin{rem}
Let $(R, \m)$ be the same as in Corollary~\ref{cor:char 0 main thm}. 
If $X=\Spec R$ is $\Q$-Gorenstein, then we can show that $\mathrm{lct}(\m) \le -a_i(R)$ for all $i$ (see the paragraph preceding Proposition~\ref{prop:a_i} for the definition of $a_i(R)$). 
The proof is as follows. 

We may assume that $i \ge 2$. 
Let $\varphi\colon Y \to X$ be the blow-up of $X$ at $\m$ and $Z=\Proj R$ be its exceptional divisor. 
Since $R$ is a normal standard graded ring, there exists a very ample divisor $H$ on $Z$ such that $R=\bigoplus_{n \ge 0} H^0(Z, \sO_Z(n H))$ and $rK_Z \sim aH$ for some $a \in \Z$, where $r$ is the Gorenstein index of $R$. 
We see by the same argument as the proof of \cite[Proposition~5.4]{SS} that $K_{Y/X}=-(1+a/r)Z$, which implies that $\mathrm{lct}(\m)$ has to be less than or equal to $-a/r$. 
Therefore, in order to prove the inequality $\mathrm{lct}(\m) \le -a_i(R)$, it suffices to show that $-a/r \le -a_i(R)$. 
This condition is equivalent to saying that if $\ell$ is an integer greater than $a/r$, then $H^{i-1}(Z, \sO_Z(\ell H))=0$, because
\[
a_i(R)=\max\{\ell \in \Z \mid H^{i-1}(Z, \sO_Z(\ell H)) \ne 0\}.
\] 
However, since $\ell H-K_Z \sim_{\Q} (\ell-a/r)H$ is ample, this is immediate from \cite[Theorem~1.7]{Fu}. 
\end{rem}

When the ring is toric, we have a similar characterization in the non-standard graded case:

\begin{cor}
Let the notation be the same as in Corollary~\ref{cor:toric charp}. 
Let $R=k[\sigma^{\vee} \cap M]$ be an affine semigroup ring over a field $k$ of characteristic zero, defined by a strongly convex rational polyhedral cone $\sigma$. Let $\m$ be the unique monomial maximal ideal of $R$. 
\begin{enumerate}[\quad\rm(1)]
\item Then $\mathrm{d}(\m)=-a_{\sigma}(R)$. 
\item $\mathrm{lct}(\m)=-a_{\sigma}(R)$ if and only if $R$ is Gorenstein. 
\end{enumerate}
\end{cor}

\begin{proof}
It follows from the existence of a toric log resolution of $\m$ that $\mathcal{J}'(\m^t)$ and $\mathcal{I}(\omega_X, \m^t)$ are torus-invariant, and so $\mathrm{lct}(\m)$ and $\mathrm{d}(\m)$ are preserved under base field extension. Thus, we may assume that $k$ is algebraically closed. 

Since (2) follows from (1), Remark~\ref{rem:graded char0}, and Corollary~\ref{cor:log terminal}, it remains to justify (1). 
For this, use the same strategy as the proof of Corollary~\ref{cor:toric charp}, in which case the assertion follows from \cite[Theorem~2]{Bl} and Lemma~\ref{lem:dbt basic}~(1). 
\end{proof}



\begin{thebibliography}{99}

\bibitem{Ab}
I.~M.~Aberbach, Extension of weakly and strongly $F$-regular rings by flat maps, J. Algebra~\textbf{241} (2001), 799--807. 

\bibitem{Bi}
C.~Birkar, Existence of flips and minimal models for $3$-folds in char $p$, Ann. Sci. \'Ec. Norm. Sup\'er. (4) \textbf{49} (2016), 169--212.

\bibitem{BCHM}
C.~Birkar, P.~Cascini, C.~D.~Hacon, and J.~$\mathrm{M^{c}}$Kernan, Existence of minimal models for varieties of log general type, J. Amer. Math. Soc.~\textbf{23} (2010), 405--468. 

\bibitem{BW} C.~Birkar and J.~Waldron, Existence of Mori fibre spaces for $3$-folds in char $p$, arXiv:1410.4511, preprint. 

\bibitem{Bl}
M.~Blickle, Multiplier ideals and modules on toric varieties, Math. Z.~\textbf{248} (2004), 113--121. 

\bibitem{BH:ainv}
W.~Bruns and J.~Herzog, On the computation of $a$-invariants, Manuscripta Math.~\textbf{77} (1992), 201--213.

\bibitem{BV}
W.~Bruns and U.~Vetter, \emph{Determinantal rings}, Lecture Notes in Math.~\textbf{1327}, Springer-Verlag, Berlin, 1988. 

\bibitem{CM}
T.~Chiba and K.~Matsuda, Diagonal $F$-thresholds and $F$-pure thresholds of Hibi rings, Comm. Algebra~\textbf{43} (2015), 2830--2851.

\bibitem{CEMS}
A.~Chiecchio, F.~Enescu, L.~E.~Miller, and K.~Schwede, Test ideals in rings with finitely generated anti-canonical algebras, arXiv:1412.6453, preprint. 

\bibitem{dFH}
T.~de~Fernex and C.~Hacon, Singularities on normal varieties, Compos. Math.~\textbf{145} (2009), 393--414. 

\bibitem{dSNB}
A.~De~Stefani and L.~N\'u\~nez-Betancourt, $F$-thresholds of graded rings, arXiv:1507.05459, preprint. 

\bibitem{Fedder}
R.~Fedder, $F$-purity and rational singularity, Trans. Amer. Math. Soc.~\textbf{278} (1983), 461--480.

\bibitem{Fu}
O.~Fujino, Fundamental theorems for semi log canonical pairs, Algebr. Geom.~\textbf{1} (2014), 194--228. 

\bibitem{FST}
O.~Fujino, K.~Schwede, and S.~Takagi, Supplements to non-lc ideal sheaves, in: \emph{Higher dimensional algebraic geometry}, RIMS K\^oky\^uroku Bessatsu B24, pp. 1--46, Res. Inst. Math. Sci. (RIMS), Kyoto,~2011.
 
\bibitem{HX} C.~D.~Hacon and C.~Xu, On the three dimensional minimal model program in positive characteristic, J. Amer. Math. Soc.~\textbf{28} (2015), 711--744. 

\bibitem{HT} N.~Hara and S.~Takagi, On a generalization of test ideals, Nagoya Math. J.~\textbf{175} (2004), 59--74.

\bibitem{HW} N.~Hara and K.-i.~Watanabe, $F$-regular and $F$-pure rings vs. log terminal and log canonical singularities, J. Algebraic Geom.~\textbf{11} (2002), 363--392. 

\bibitem{HY} N.~Hara and K.~Yoshida, A generalization of tight closure and multiplier ideals, Trans. Amer. Math. Soc.~\textbf{355} (2003), 3143--3174.

\bibitem{HV} I.~B.~Henriques and M.~Varbaro, Test, multiplier and invariant ideals, Adv. Math.~\textbf{287} (2016), 704--732.

\bibitem{HWY} D.~Hirose, K.-i.~Watanabe, and K.~Yoshida, $F$-thresholds versus $a$-invariants for standard graded toric rings, Comm. Algebra~\textbf{42} (2014), 2704--2720. 

\bibitem{HuTr} C.~Huneke and N.~V.~Trung, On the core of ideals, Compositio Math.~\textbf{141} (2005), 1--18.

\bibitem{HS} E.~Hyry and K.~E.~Smith, On a non-vanishing conjecture of Kawamata and the core of an ideal, Amer. J. Math.~\textbf{125} (2003), 1349--1410.

\bibitem{Ka} M.~Kawakita, Inversion of adjunction on log canonicity, Invent. Math.~\textbf{167} (2007), 129--133. 

\bibitem{Ko} J.~Koll\'ar, Exercises in the birational geometry of algebraic varieties, arXiv:0809.2579, preprint. 

\bibitem{KM} J.~Koll\'ar and S.~Mori, \emph{Birational geometry of algebraic varieties}, Cambridge Tracts in Mathematics~\textbf{134}, Cambridge University Press, Cambridge, 1998. 

\bibitem{KSS} S.~J.~Kov\'acs, K.~Schwede, and K.~E.~Smith, The canonical sheaf of Du Bois singularities, Adv. Math.~\textbf{224} (2010), 1618--1640. 

\bibitem{Ma} L.~Ma, $F$-injectivity and Buchsbaum singularities, Math. Ann~\textbf{362} (2015), 25--42.

\bibitem{MSV} L.~E.~Miller, A.~K.~Singh, and M.~Varbaro, The $F$-pure threshold of a determinantal ideal, Bull. Braz. Math. Soc.~\textbf{45} (2014), 767--775.

\bibitem{Sch1}
K.~Schwede, A simple characterization of Du Bois singularities. Compos. Math.~\textbf{143} (2007), 813--828. 

\bibitem{Sch}
K.~Schwede, Generalized test ideals, sharp $F$-purity, and sharp test elements, Math. Res. Lett.~\textbf{15} (2008), 1251--1261. 

\bibitem{SS}
K.~Schwede and K.~E.~Smith, Globally $F$-regular and log Fano varieties, Adv. Math.~\textbf{224} (2010), 863--894. 

\bibitem{ST}
K.~Schwede and S.~Takagi, Rational singularities associated to pairs, Michigan Math. J.~\textbf{57} (2008), 625--658. 

\bibitem{Sh}
K.~Shibata, \emph{Rational singularities, $\omega$-multiplier ideals and cores of ideals}, Ph.D. thesis, University of Tokyo, 2016.

\bibitem{Song}
Q.~Song, \emph{Questions in local cohomology and tight closure}, Ph.D. thesis, University of Utah, 2007.
 
\bibitem{SrT}
V.~Srinivas and S.~Takagi, Nilpotence of Frobenius action and the Hodge filtration on local cohomology, arXiv:1503.08772, preprint. 

\bibitem{Ta}
S.~Takagi, $F$-singularities of pairs and inversion of adjunction of arbitrary codimension, Invent. Math.~\textbf{157} (2004), 123--146.

\bibitem{TW}
S.~Takagi and K.-i.~Watanabe, On $F$-pure thresholds, J. Algebra~\textbf{282} (2004), 278--297.

\end{thebibliography}
\end{document}